\documentclass[11pt,a4paper,oneside]{article}
\usepackage{amsfonts}
\usepackage{graphicx}
\usepackage{amsmath,amssymb}
\usepackage{color}
\usepackage{amssymb}
\usepackage{amsmath}
\usepackage[T1]{fontenc}
\usepackage{amsthm}
\usepackage[margin=1in]{geometry}

\usepackage{enumerate}
\usepackage{hyperref}

\providecommand{\U}[1]{\protect\rule{.1in}{.1in}}
\newtheorem{theorem}{Theorem}
\newtheorem{lemma}[theorem]{Lemma}
\newtheorem{proposition}[theorem]{Proposition}

\newtheorem{corollary}[theorem]{Corollary}

\theoremstyle{definition}

\newtheorem{example}[theorem]{Example}

\theoremstyle{remark}
\newtheorem{remark}[theorem]{Remark}

\begin{document}

\title{Asymptotic variance of stationary reversible and normal Markov
processes}
\author{ George Deligiannidis \thanks{%
Department of Statistics, University of Oxford, OX1 3TG, UK; \texttt{%
deligian@stats.ox.ac.uk}} \and Magda Peligrad\thanks{
Department of Mathematical Sciences, University of Cincinnati, PO Box
210025, Cincinnati, Oh 45221-0025, USA. \texttt{peligrc@ucmail.uc.edu,
peligrm@ucmail.uc.edu}} \thanks{%
Supported in part by a Charles Phelps Taft Memorial Fund grant, and the NSF
grant DMS-1208237. } \and Sergey Utev \thanks{
Department of Mathematics, University of Leicester, University Road, LE1 7RH
UK, \texttt{\ su35@leicester.ac.uk} }}
\maketitle

\begin{abstract}
We obtain necessary and sufficient conditions for the regular variation of
the variance of partial sums of functionals of discrete and continuous-time
stationary Markov processes with normal transition operators. We also
construct a class of Metropolis-Hastings algorithms which satisfy a central
limit theorem and invariance principle when the variance is not linear in $n$%
. 
\end{abstract}

\section{Introduction}

Let $({\xi}_{n})_{n\in\mathbb{Z}}$ be a stationary Markov chain defined on a
probability space $(\Omega,\mathcal{F},\mathbb{P})$ with values in a general
state space $(S,\mathcal{A})$ and let the marginal distribution be denoted
by $\pi(A)=\mathbb{P}({\xi}_{0}\in A)$. We assume that there is a regular
conditional distribution denoted by $Q(x,A)=\mathbb{P}({\xi}_{1}\in A|\,{\xi 
}_{0}=x)$. Let $Q$ also denote the Markov transition operator acting via $%
(Qg)(x)=\int_{S}g(s)Q(x,\mathrm{d}s)$, on $\mathbb{L}_{0}^{2}(\pi)$, the set
of measurable functions on $S$ such that $\int_{S}g^{2}(s)\pi(\mathrm{d}%
s)<\infty$ and $\int_{S}g(s)\pi(\mathrm{d}s)=0.$ If $g, h\in \mathbb{L}
_{0}^{2}(\pi)$, the integral $\int_{S}g(s)h(s)\pi(\mathrm{d}s)$ will
sometimes be denoted by $\langle g,h\rangle$.

For some function $g \in \mathbb{L}_{0}^{2}(\pi)$, let 
\begin{equation}
X{_{i}=g(\xi}_{i}{),\quad S_{n}(}X{)=\sum\limits_{i=1}^{n}}X_{i},
\quad\sigma _{n}({g)}=(\mathbb{E}S_{n}^{2}(X{)})^{1/2}.   \label{defcsi}
\end{equation}
{\ Denote by $\mathcal{F}_{k}$ the $\sigma$--field generated by $\xi_{i}$
with $i\leq k$. }

For any integrable random variable $X$ we denote by $\mathbb{E}_{k}X=\mathbb{%
E}(X|\mathcal{F}_{k}).$ With this notation, $\mathbb{E}_{0}X_{1}=Qg(${$\xi$}$%
_{0})=\mathbb{E}(X_{1}|${$\xi$}$_{0}).$ We denote by ${{\|X\|}_{p}}$ the
norm in {$\mathbb{L}^{p}$}$(\Omega,\mathcal{F},\mathbb{P}).$

The Markov chain is called \textit{normal} when the transition operator $Q$
is  normal, that is it commutes with its adjoint $Q^{\ast},$ namely $%
QQ^{\ast}=Q^{\ast}Q$.

From the spectral theory of normal operators on Hilbert spaces (see for
instance \cite{R}), it is well known that for every $g\in L_{0}^{2}(\pi )$
there is a unique \emph{transition spectral measure} $\nu $ supported on the
spectrum of the operator $D:=\{z\in \mathbb{C}:|z|\leq 1\}$, such that 
\begin{equation}
\mathrm{cov}(X_{0},X_{n})=\mathrm{cov}((g({\xi }_{0}),Q^{n}g({\xi }%
_{0}))=\langle g,Q^{n}g\rangle =\int_{D}z^{n}\nu (\mathrm{d}z).  \label{cov}
\end{equation}%
and 
\begin{equation*}
\mathrm{cov}(\mathbb{E}_{0}(X_{i}),\mathbb{E}_{0}(X_{j}))=\langle
Q^{i}g,Q^{j}g\rangle =\langle g,Q^{i}(Q^{\ast })^{j}g\rangle =\int_{D}z^{i}%
\bar{z}^{j}\nu (\mathrm{d}z).
\end{equation*}

In particular, the Markov chain is reversible if $Q=Q^{\ast}.$ The condition
of reversibility is equivalent to requiring that $(${$\xi$}$_{0},${$\xi$}$%
_{1})$ and $({\xi}_{1},{\xi}_{0}) $ have the same distribution. Furthermore,
in the reversible case $\nu$ is concentrated on $[-1,1]$.

Limit theorems for additive functionals of reversible Markov chains have
received considerable attention in the literature not only for their
intrinsic interest, but also for their great array of applications which
range from interacting particle systems (see the seminal paper by Kipnis and
Varadhan \cite{KV}) and random walks in random environments (see for example 
\cite{Toth}), to the relatively recent applications in computational
statistics with the advent of Markov Chain Monte Carlo algorithms (e.g.\ 
\cite{HaggRos,RR}). Limit theorems have appeared under a great array of
conditions, notably geometric ergodicity (for an overview see \cite{KM}),
conditions on the growth of the conditional expectations $\mathbb{E}%
(S_{n}|X_{1})$ (see e.g.\ \cite{MW}, \cite{PU}), under spectral conditions
(see \cite{KV,GL,DL}) or under conditions on the resolvent of the transition
operator (see \cite{Toth}), a method which is also applicable in the
non-normal case where spectral calculus may not apply.

The variance of the partial sums plays a major role in limit theorems where
it acts as a normalizer, and also in computational statistics where the
asymptotic variance is used as a measure of the efficiency of an algorithm
(see e.g.\ \cite{Tierney}). It is not surprising then, that in certain
cases, conditions for the central limit theorem have been imposed directly
on the growth of the variance. In fact in 1986 Kipnis and Varadhan~\cite{KV}
proved the functional form of the central limit theorem for functionals of
stationary reversible ergodic Markov chains under the assumption that 
\begin{equation}
\lim_{n\rightarrow \infty }\frac{\mathrm{var}(S_{n})}{n}=\sigma _{g}^{2},
\label{condvar}
\end{equation}%
and further established necessary and sufficient conditions for the variance
of the partial sums to behave linearly in $n$ in terms of the transition
spectral measure $\nu $. In particular they showed that for any reversible
ergodic Markov chain the convergence in \eqref{condvar} is sufficient for
the functional central limit theorem $S_{[nt]}/\sqrt{n}\Rightarrow |\sigma
_{g}|W(t)$ (where $W(t)$ is the standard Brownian motion, $\Rightarrow $
denotes weak convergence and $[x]$ denotes the integer part of $x$).
Moreover, \eqref{condvar} is equivalent to the fact that the finite limiting
variance is then given by 
\begin{equation}
\sigma _{g}^{2}=\int\nolimits_{-1}^{1}\frac{1+t}{1-t}\nu (\mathrm{d}%
t)<\infty .  \label{SR}
\end{equation}%
Furthermore, according to Remark 4 on page 514 in \cite{DL} if in addition
to (\ref{SR}) we assume $\rho (-1)=0$\ then, we also have 
\begin{equation*}
\lim_{n\rightarrow \infty }{\sum\limits_{i=0}^{n}}\mathrm{cov}%
(X_{0},X_{i})=\int\nolimits_{-1}^{1}\frac{1}{1-t}\nu (\mathrm{d}t).
\end{equation*}%
See also \cite{HaggRos} for a discussion of when \eqref{condvar} and %
\eqref{SR} are equivalent.

It is remarkable that in the reversible case, conditions \eqref{condvar} and %
\eqref{SR} are equivalent, both sufficient for the central limit theorem and
invariance principle, and conjectured to be sufficient for the almost sure
conditional central limit theorem.%
It is an open problem, whether \eqref{SR} is also necessary for the
central limit theorem.

On the other hand, notice that any invertible transformation $T$ generates a
unitary, and thus normal, transition operator $Qf(x)=f(T(x))$, since $%
Q^{\ast }f(x)=f(T^{-1}(x))$ whence $QQ^{\ast }=Q^{\ast }Q=I$ is the identity
operator. In particular, any stationary sequence $\xi _{i}$, can be treated
as a functional of a normal Markov chain. Therefore for normal,
non-reversible Markov chains, \eqref{condvar} and the central limit theorem
and invariance principle are no longer equivalent without 
 further assumptions (see e.g. Bradley \cite{Bra} and Giraudo and Voln%
\'{y}~\cite{GiVo} for counterexamples).

For the non-reversible case, Gordin and Lif\v{s}ic \cite{GL} applied
martingale methods and stated, among other results, the central limit
theorem for functionals of stationary ergodic Markov chains with normal
transition operator, under the spectral condition 
\begin{equation}
\int_{D}\frac{1}{|1-z|}\nu (\mathrm{d}z)<\infty .  \label{eq:GLcond}
\end{equation}%
If condition (\ref{eq:GLcond}) holds then (\ref{condvar}) also holds with 
\begin{equation}
\sigma ^{2}:=\int_{D}\frac{1-|z|^{2}}{|1-z|^{2}}\nu (\mathrm{d}z)<\infty .
\label{sigma}
\end{equation}%
One of our main results, Theorem \ref{pr:NSC}, gives necessary and
sufficient conditions for the existence of the limit $\mathrm{var}%
(S_{n})/n\rightarrow K<\infty $. We shall see 
that $\mathrm{var}(S_{n})/n\rightarrow K$ if and only if $\sigma ^{2}<\infty 
$ and $\nu (U_{x})/x\rightarrow C$ as $x\rightarrow 0^{+}$, where 
\begin{equation*}
U_{x}=\{z=(1-r)e^{iu}\;:|z|\leq 1,\;0\leq r\leq |u|\leq x\}.
\end{equation*}%
In this case $K=\sigma ^{2}+{\pi C}$. Furthermore if (\ref%
{eq:GLcond}) holds then $C=0.$

Recently Zhao et al.~\cite{ZWV} and Longla et al.~\cite{LPP}, in the context
of reversible Markov chains, studied the asymptotic behavior of $S_{n}$ for
the more general case when the variance of partial sums behaves as a
regularly varying function $\sigma _{n}^{2}=\mathrm{var}(S_{n})=nh(n)$ where 
$h(x)$ is slowly varying, i.e. $h:(0,\infty )\rightarrow (0,\infty ),$
continuous, and $h(st)/h(t)\rightarrow 1$ for all $s>0$. For this case the
situation is different and in \cite{ZWV,LPP} examples are given of
stationary, reversible, and ergodic Markov chains that satisfy the CLT under
a normalization different of $\sigma _{n},$ namely $S_{n}/\sigma
_{n}\Rightarrow N(0,c^{2})$ for a $c\neq 1$ and $c\neq 0$.

On the other hand, in a recent paper, Deligiannidis and Utev \cite{DU} have
studied the relationship between the variance of the partial sums of weakly
stationary processes and the spectral measure induced by the unitary shift
operator. To be more precise, by the Birghoff-Herglotz Theorem (see e.g.\
Brockwell and Davis~\cite{BD}), there exists a unique measure on the unit
circle, or equivalently a non-decreasing function $F,$ called the \emph{%
spectral distribution function} on $[0,2\pi ]$, such that 
\begin{equation}
\mathrm{cov}(X_{0},X_{n})=\int_{0}^{2\pi }\mathrm{e}^{\mathrm{i}n\theta }F(%
\mathrm{d}\theta ),\quad \text{for all}\,\,n\in \mathbb{Z}\,.  \label{SpM}
\end{equation}%
If $F$ is absolutely continuous with respect to the normalized Lebesgue
measure $\lambda $ on $[0,2\pi ]$, then the Radon-Nikodym derivative $f$ of $%
F$ with respect to the Lebesgue measure is called the \textit{spectral
density;} in other words $F(\mathrm{d}\theta )=f(\theta )\mathrm{d}\theta $%
,. 
The main result of \cite{DU} is given below. In the sequel, the notation $%
a_{n}\sim b_{n}$ as $n\rightarrow \infty $ means that $\lim_{n\rightarrow
\infty }a_{n}/b_{n}=1.$

\vskip10pt \noindent\textbf{Theorem A}. [Deligiannidis and Utev \cite{DU}]%
\label{thm:DU} \textit{Let $S_{n}:=X_{1}+\cdots+X_{n}$ where $(X_{i})_{i\in%
\mathbb{Z}}$ is a {real} weakly stationary sequence. For $%
\alpha\in(0,2)$, define $C(\alpha):=\Gamma(1+\alpha)\sin(\tfrac{\alpha\pi}{2}%
)/[\pi(2-\alpha )]$, and let $h$ be slowly varying at infinity. Then $%
\mathrm{var}(S_{n})\sim n^{\alpha}h(n)$ as $n\rightarrow\infty$ if and only
if $F(x)\sim {\frac{1}{2}}C(\alpha)x^{2-\alpha}h(1/x)$ as $%
x\rightarrow0$. } \vskip10pt

In this paper we obtain necessary and sufficient conditions for the regular
variation of the variance of partial sums of functionals of stationary
Markov chains with normal operators. The necessary and sufficient conditions
are based on several different representations in terms of: 

\begin{enumerate}
\item the spectral distribution function in the sense of the
Birghoff-Herglotz theorem,

\item the transition spectral measure of the associated transition operator,

\item the harmonic measure of Brownian motion in the disk,

\item a martingale decomposition.
\end{enumerate}


In the case of stationary reversible Markov Chains we also construct a class
of Metropolis-Hastings algorithms with non-linear growth of variance, for
which we establish the invariance principle and conditional central limit
theorem with normalization $\sqrt{nh(n)}$. \vskip3pt \noindent \textit{%
Continuous-time processes.} In the continuous time setting, let $\{\xi
_{t}\}_{t\geq 0}$ be a stationary Markov process with values in the general
state space $(S,\mathcal{A})$, defined on a probability space $(\Omega ,%
\mathcal{F},\mathrm{P})$, with stationary measure $\pi $. 
We assume that the contraction semigroup 
\begin{equation*}
T_{t}g(x):=\mathbb{E}[g(\xi _{t})|\xi _{0}=x],\quad g\in L^{2}(\pi ),\quad
t\geq 0,
\end{equation*}%
is strongly continuous on $L^{2}(\pi )$, and we let $\{\mathcal{F}%
_{t}\}_{t\geq 0}$ be a filtration on $(\Omega ,\mathcal{F},\mathrm{P})$ with
respect to which $\{\xi _{t}\}_{t}$ is progressively measurable and
satisfies the Markov property $\mathrm{E}(g(\xi _{t})|\mathcal{F}%
_{u})=T_{t-u}g(\xi _{u})$, for any $g\in L^{2}(\pi )$ and $0\leq u<t$.
Furthermore we can write $T_{t}=\mathrm{e}^{Lt}$, where $L$ is the
infinitesimal generator of the process $\{\xi _{t}\}_{t}$, and $\mathcal{D}%
(L)$ its domain in $L^{2}(\pi )$. We assume $T_{t}$ to be normal, that is $%
T_{t}^{\ast }=T_{t}$, which then implies that $L$ is a normal, possibly
unbounded operator, with spectrum supported in the left half-plane $\{z\in 
\mathbb{C}:\Re (z)\leq 0\}$ (see \cite[Theorem~13.38]{R}). In the reversible
case the spectrum of $%
L$ is supported on the left real half-axis(see \cite[Remark~1.7]{KV}).

Similarly to the discrete case, with any $f\in L^2(\mu)$ we can associate a
unique spectral measure $\nu(\mathrm{d} z) = \nu_f (\mathrm{d} z)$ supported
on the spectrum of $L$ such that 
\begin{equation*}
\langle f, T_t f\rangle = \mathrm{cov}( f(\xi_0), f(\xi_t)) =
\int_{\Re(z)\leq 0} \mathrm{e} ^{z t} \nu (\mathrm{d} z).
\end{equation*}

In the reversible case Kipnis and Varadhan~\cite{KV} proved an invariance
principle under the condition that $f\in \mathcal{D}\big( (-L)^{-1/2}\big)$,
which in spectral form is equivalent to 
\begin{equation}
\int_{\alpha=-\infty}^{0}\frac{-1}{\alpha}\nu(\mathrm{d}\alpha)<\infty . 
\label{eq:1/a}
\end{equation}
Building on the techniques in \cite{GL,KV}, Holzmann~\cite{HH,HH2}
established the central limit theorem for processes with normal transition
semi-groups (see also \cite{O}), under the condition 
\begin{equation}
\int_{\Re(z)\leq0}|z|^{-1}\nu(\mathrm{d}z)<\infty.   \label{eq:1/z}
\end{equation}
In this case 
\begin{equation*}
\lim_{T\rightarrow\infty}\frac{\mathrm{var}(S_{T})}{T}= -2\int_{\mathbb{H}%
^{-}}\Re(1/z)\nu(\mathrm{d}z)=:\varsigma^{2}. 
\end{equation*}
On the other hand, using resolvent calculus, Toth~\cite{Toth2,Toth} treated
general discrete and continuous-time Markov processes and obtained a
martingale approximation, central limit theorem and convergence of
finite-dimensional distributions to those of Brownian motion, under
conditions on the resolvent operator which may hold even in the non-normal
case. Similar conditions, albeit in the normal case, also appeared later in 
\cite{HH,HH2}.

Under any of the above conditions, it is clear that the variance of $S_{T}$
is asymptotically linear in $T$. Similarly to the discrete case, we show in
Theorem~\ref{pr:NSCcts}, that $\mathrm{var}(S_n)/n \to K = \varsigma^2 + \pi
C$ if and only if $\varsigma^2<\infty$ and $\nu(U_x)/x\to C$, where 
\begin{equation*}
U_{x} = \{ a+ \mathrm{i} b: 0\leq-a\leq|b|\leq x\}. 
\end{equation*}

The rest of the paper is structured as follows. We provide our results for
discrete time processes in Section 2 and for continuous time in Section 3.
Section 4 contains the proofs, while the Appendix contains two standard
Tauberian theorems to make the text self-contained, and technical lemmas
used in Section 3.

\section{Results for Markov chains}

\subsection{Relation between the transition spectral measure and spectral
distribution function}

Our first result gives a representation of the spectral distribution
function in terms of the transition spectral measure. This link makes
possible to use the results in \cite{DU} to analyze the variance of partial
sums. Quite remarkably, if the transition spectral measure is supported on
the open unit disk, the spectral distribution function is absolutely
continuous with spectral density given by \eqref{eq:normaldensity}, and in
this case the sequence $\mathrm{cov}(X_{0},X_{n})$ converges to $0$.

\begin{lemma}[Representation Lemma]
\label{thm:normaldensity} Let $({\xi}_{n})_{n\in \mathbb{Z}}$ be a
stationary Markov chain, with normal transition operator $Q$. Let $g\in
L_{0}^{2}(\pi)$, $X_{i}:=g(\xi_{i})$ and write $\nu=\nu_{g}$ for the
operator spectral measure with respect to $g$. Also denote the unit circle $%
\Gamma:=\{z:|z|=1\}$ and by $D_{0}:=\{z:|z|<1\}.$ Denote by $\nu_{\Gamma}$
the restriction of measure $\nu$ to $\Gamma$ and by $\nu_{0}$ denote the
restriction of measure $\nu$ to $D_{0}.$ Then 
\begin{equation*}
\mathrm{cov}(X_{0},X_{n})=\int_{0}^{2\pi}\mathrm{e}^{\mathrm{i}tn}[\nu
_{\Gamma}(\mathrm{d}t)+f(t)\mathrm{d}t], 
\end{equation*}
where 
\begin{equation}
f(t)={\frac{1}{2\pi}}\int_{D_{0}}\frac{1-|z|^{2}}{|1-z\mathrm{e}^{\mathrm{i}%
t}|^{2}}\nu_{0}(\mathrm{d}z).   \label{eq:normaldensity}
\end{equation}
Furthermore the spectral distribution function has the representation 
\begin{equation}
F(\mathrm{d}t)=\nu_{\Gamma}(\mathrm{d}t)+f(t)\mathrm{d}t.   \label{SMR}
\end{equation}
\end{lemma}

\begin{remark}
By integrating relation (\ref{SMR}) we obtain%
\begin{equation}
F(x)=x\int_{D_{0}}T_{x}(z)\nu_{0}(\mathrm{d}z)+\nu_{\Gamma}([0,x]) 
\label{SpMe}
\end{equation}
where 
\begin{equation*}
T_{x}(z):=\frac{1}{2\pi}(1-|z|^{2})\int_{0}^{1}\frac{dt}{|1-z\mathrm{e}^{%
\mathrm{i}tx}|^{2}}\text{.}
\end{equation*}
\end{remark}

\noindent By combining Representation Lemma \ref{thm:normaldensity} with
Theorem A we obtain the following corollary.

\begin{corollary}
\label{cor5}Let $({\xi}_{n})_{n\in\mathbb{Z}}$ be as in Lemma \ref%
{thm:normaldensity} and let $\alpha\in(0,2).$ Then $\mathrm{var}%
(S_{n})=n^{\alpha}h(n)$ as $n\rightarrow\infty$ if and only if $F(x)=%
{\frac{1}{2}}{C(\alpha
)}x^{2-\alpha}h(1/x)$ as $x\rightarrow0^{+}.$
\end{corollary}

It should be obvious from the statement of Representation Lemma \ref%
{thm:normaldensity} that Theorem A is directly applicable to the measure $%
\mathrm{d} F$. The conditions on $F,$ mentioned in Corollary \ref{cor5},
when expressed in terms of the operator spectral measure, become technical
conditions on the growth of integrals of the Poisson kernel over the unit
disk. To get further insight into this lemma we shall apply it to reversible
Markov chains.

Our next result, a corollary of Representation Lemma \ref{thm:normaldensity}
combined with Theorem A, provides this link and points out a set of
equivalent conditions for regular variation of the variance for reversible
Markov chains. In fact, as it turns out, if the spectral measure has no
atoms at $\pm1$, it follows that the spectral distribution function is
absolutely continuous and we obtain an expression for the spectral density.
Related ideas, under more restrictive assumptions have appeared in \cite{JB}%
, while in \cite{DL} a spectral density representation was obtained for
positive self-adjoint transition operators, in other words $\nu$ supported
on $[0,1)$.

\begin{corollary}
\label{prop:SpDens} Assume that $Q$ is self-adjoint and that the transition
spectral measure $\nu$ does not have atoms at $\pm1.$ Then, the spectral
distribution function $F$ defined by (\ref{SpM}) is absolutely continuous
with spectral density given by 
\begin{equation}
f(t)={\frac{1}{2\pi}}\int_{-1}^{1}\frac{1-\lambda^{2}}{1+\lambda^{2}-2%
\lambda\cos t}\mathrm{d}\nu(\lambda),   \label{eq:revdensity}
\end{equation}
and for $\alpha\in\lbrack1,2),$ the following are equivalent:

\begin{description}
\item[(i)] $\mathrm{var}(S_{n})=n^{\alpha}h(n)\text{ as }n\rightarrow\infty, 
$

\item[(ii)] $F(x)={\frac{1}{2}}{C(\alpha)}x^{2-\alpha}h(1/x)%
\text{ as }x\rightarrow0^{+}. $
\end{description}

Moreover, if $h(x)\rightarrow\infty$ as $x\rightarrow0^{+}$, then (i), (ii)
are equivalent to

$\displaystyle \int_{0}^{1}\frac{r(\mathrm{d}y)}{x^{2}+y^{2}}={\frac{\pi}{%
{2}} C(\alpha )}x^{1-\alpha}h(1/x)+O(1)\text{ as }
x\rightarrow0^{+}$; \thinspace where $r(0,y]=\nu(1-y,1)$.
\end{corollary}

\subsection{ Relation between spectral measure and planar Brownian motion}

Our next result makes essential use of the Poisson kernel which appears in %
\eqref{eq:normaldensity} to provide a fascinating interpretation of the
spectral distribution function $F$ in terms of the harmonic measure of
planar Brownian motion started at a random point in the open unit disk.

\begin{theorem}
\label{thm:harmonicmeasure} Let $\nu$ be the transition spectral measure,
and let $(B_{t}^{z})_{t\geq0}$ be standard planar Brownian motion in $%
\mathbb{C}$, started at the point $z\in D$. Also let $Z$ be a random point
in $D$ distributed according to $\nu$ and let $\tau_{D}^{Z}:=\inf\{t {\geq}
0:B_{t}^{Z}\notin D\}$. Let $\Gamma_{x}:= \{z: z=\mathrm{e}^{\mathrm{i}t y},
|y|<x\}$ and $\alpha\in(0,2).$ Then, the following statements are equivalent:

\begin{description}
\item[(i)] $\mathrm{var}(S_{n})\sim n^{\alpha}h(n)$ as $n\rightarrow \infty$;

\item[(ii)] $\mathrm{P}\big\{B_{\tau_{D}^{Z}}^{Z}\in\Gamma _{x}\big\}\sim
C(\alpha)x^{2-\alpha}h(1/x)/\nu(D)$ as $x\rightarrow0$.
\end{description}
\end{theorem}


\subsection{Linear growth of variance for partial sums for normal Markov
Chains}

By applying martingale techniques we establish necessary and sufficient
conditions for the asymptotic linear variance behavior for general normal
Markov chains. 

\begin{theorem}
\label{pr:NSC} With the notation of Lemma~\ref{thm:normaldensity}

\begin{description}
\item[(a)] The limit, $\mathrm{var}(S_{n})/n\rightarrow K<\infty$ exists if
and only if 
\begin{align}
\sigma^{2}&:=\int_{D}\frac{1-|z|^{2}}{|1-z|^{2}}\nu(\mathrm{d}z)<\infty,
\qquad \text{and}  \label{normalsigma} \\
I_{n}&:=\frac{1}{n}\int_{D}\frac{|1-z^{n}|^{2}}{|1-z|^{2}}\nu(\mathrm{d}%
z)\rightarrow L,   \label{new_lin}
\end{align}
where $K=\sigma^{2}+L$ $.$

\item[(b)] Moreover, under (\ref{normalsigma}) the following are equivalent:

\begin{description}
\item[(i)] (\ref{new_lin}) holds with $L={\pi C}$\textit{.}

\item[(ii)] $\nu (U_{x})/x\rightarrow C$\textit{\ as }$x\rightarrow 0^{+}$%
\textit{, where }%
\begin{equation*}
U_{x}=\{z=(1-r)e^{\mathrm{i}u}\in D\;:\;0\leq r\leq |u|\leq x\}.
\end{equation*}

\item[(iii)] $n\nu (D_{n})\rightarrow C$\textit{\ as }$x\rightarrow 0^{+}$%
\textit{, where }%
\begin{equation*}
D_{n}=\{z=re^{2\mathrm{i}\pi \theta };\text{ }1-\frac{1}{n}\leq r\leq 1,%
\text{ }-\frac{1}{n}\leq \theta \leq \frac{1}{n}\}.
\end{equation*}
\end{description}
\end{description}
\end{theorem}

\vskip10pt

It should be noted that there are many sufficient conditions for the
convergence 
\begin{equation}
\mathrm{var}(S_{n})/n\rightarrow\sigma^{2}<\infty.   \label{popular}
\end{equation}

\noindent(1) It is immediate from the proof of Theorem \ref{pr:NSC}, that (%
\ref{popular}) is equivalent to $\sigma^{2}<\infty$ and 
\begin{equation}
\frac{1}{n}\int_{D}\frac{|1-z^{n}|^{2}}{|1-z|^{2}}\nu(\mathrm{d}%
z)\rightarrow0.   \label{cl_linear:LG}
\end{equation}
\noindent(2) In Corollary 7.1 in \cite{BI} it was shown that if we assume 
\begin{equation*}
\int_{D}\frac{1}{|1-z|}\nu(\mathrm{d}z)<\infty, 
\end{equation*}
then both (\ref{normalsigma}) and (\ref{cl_linear:LG}) are satisfied and so
convergence (\ref{popular}) holds, the result attributed to Gordin and Lif%
\v{s}ic \cite{GL} (see Theorem 7.1 in \cite{BI}; see also \cite{DL}).

\noindent(3) From Representation Lemma \ref{thm:normaldensity} and Ibragimov
version of Hardy-Littlewood theorem 
\begin{equation*}
\mathrm{var}(S_{n})/n\rightarrow2\pi f(0)=\sigma^{2}. 
\end{equation*}

\noindent(4) On the other hand, from the Representation Lemma \ref%
{thm:normaldensity} and Theorem A, convergence (\ref{popular}) is equivalent
to the uniform integrability of $T_{x}(z)$ with respect to $\nu _{0}$ as $%
x\rightarrow0^{+}$.

\noindent(5) Motivated by the complex Darboux-Wiener-Tauberian approach
(e.g. as in \cite{DU11}), by analyzing 
\begin{equation*}
V(\lambda)=\sum_{n=1}^{\infty}\mathrm{var}(S_{n})\lambda^{n}\;=\frac{\lambda 
}{(1-\lambda)^{2}}\int_{D}\frac{1+\lambda z}{1-\lambda z}\nu(dz), 
\end{equation*}
a sufficient condition for (\ref{popular}) is 
\begin{equation*}
\int_{D}\frac{1+\lambda z}{1-\lambda z}\nu(dz)\rightarrow\sigma^{2}=\int _{D}%
\frac{1+z}{1-z}\nu(dz)\text{ as }\lambda\rightarrow1\text{ with }%
|\lambda|<1. 
\end{equation*}
Here, since $\nu(\mathrm{d}z)=\nu(\mathrm{d}\bar{z})$, the integral is
understood in the Cauchy sense i.e. 
\begin{equation*}
\int_{D}f(z)\nu(dz)=\frac{1}{2}\int_{D}[f(z)+f(\bar{z})]\nu(dz). 
\end{equation*}

\noindent (6) From Theorem \ref{pr:NSC}, it follows that \eqref{popular} is
equivalent to $\sigma ^{2}<\infty $ and $\nu (U_{x})/x\rightarrow 0$ as $%
x\rightarrow 0^{+}$.

\noindent(7) Finally, again from Theorem \ref{pr:NSC}, it follows that %
\eqref{popular} is equivalent to $\sigma^{2}<\infty$ and $%
n\nu(D_{n})\rightarrow0$ as $n\rightarrow\infty$. This result is also a
corollary of the following inequality motivated by Cuny and Lin~\cite{CL} 
\begin{equation*}
\frac{n\nu(D_{n})}{36}\leq\frac{1}{n}\mathbb{E}(\mathbb{E}%
_{1}(S_{n}))^{2}\leq\frac{4}{n}{\displaystyle\sum\limits_{j=1}^{n-1}}%
j\nu(D_{j}). 
\end{equation*}

\begin{remark}
Notice that when the transition spectral measure $\nu$ is concentrated on $%
\Gamma$ (dynamical system), then $\sigma^{2}=0$ and so $\sigma^{2}$ cannot
be the limiting variance, in general.
\end{remark}

By inspecting the proof of the Theorem \ref{pr:NSC}, under %
\eqref{normalsigma} we are able to characterize regular variation of $%
\mathrm{var}(S_{n})$ when $\liminf_{n\rightarrow \infty }\mathrm{var}%
(S_{n})/n>0$. More exactly, we have the following proposition.

\begin{proposition}
Assume $\liminf_{n\rightarrow \infty }\mathrm{var}(S_{n})/n>0$ and that %
\eqref{normalsigma} holds. Then, for $\alpha \in \lbrack 1,2)$, and a
positive function $h$, slowly varying at infinity, $\mathrm{var}(S_{n})\sim
n^{\alpha }h(n)$ as $n\rightarrow \infty $ if and only if $\nu (U_{x})\sim
C(\alpha )x^{2-\alpha }h(1/x)$ as $x\rightarrow 0^{+}$, with $C(\alpha )$ as
defined in Theorem~\textsl{A}. In particular, $\mathrm{var}(S_{n})/n$ is
slowly varying at $n\rightarrow \infty $ if and only if $\nu (U_{x})/x$ is
slowly varying as $x\rightarrow 0^{+}$.
\end{proposition}


\subsection{Relation between the variance of partial sums and transition
spectral measure of reversible Markov chains}

We continue the study of stationary reversible Markov chains and provide
further necessary and sufficient conditions for its variance to be regularly
varying, in terms of the operator spectral measure by a direct approach,
without the link with the spectral distribution function.

\begin{theorem}
\label{thR(x)}Assume $Q$ is self-adjoint, $\alpha \geq 1$, $\mathrm{var}%
(S_{n})/n\rightarrow \infty $, and let $c_{\alpha }:=\alpha (2-\alpha
)/2\Gamma (3-\alpha )$. Then 
\begin{equation*}
V(x)=\int\nolimits_{-1}^{1-x}\frac{1}{1-t}\nu (\mathrm{d}t)\sim c_{\alpha
}x^{1-\alpha }h(\frac{1}{x})\text{ as }x\rightarrow 0^{+}
\end{equation*}%
if and only if 
\begin{equation*}
\mathrm{var}(S_{n})=n^{\alpha }h(n)\text{ as }n\rightarrow \infty .
\end{equation*}%
Furthermore if $\alpha >1$ then $\mathrm{var}(S_{n})=n^{\alpha }h(n)$ as $%
n\rightarrow \infty $ iff 
\begin{equation*}
\nu (1-x,1]\sim d_{\alpha }x^{2-\alpha }h(\frac{1}{x})\text{ as }%
x\rightarrow 0^{+},
\end{equation*}%
where $d_{\alpha }:=\alpha (\alpha -1)/2\Gamma (3-\alpha )$.
\end{theorem}

\begin{remark}
It should be obvious from the statement in the above theorem, that regular
variation of the variance is equivalent to regular variation of the
transition spectral measure only in the case $\alpha>1$. As the following
example demonstrates, in the case $\alpha=1$, there are reversible Markov
chains whose variance of partial sums varies regularly with exponent $1$
even though $\nu(1-x,1]$ is not a regularly varying function.
\end{remark}

\begin{example}
Take a probability measure $\upsilon$ on $[-1,1]$ defined for $0<a<1/2$ by 
\begin{equation}
\mathrm{d}\upsilon=\frac{1}{c}(1-|x|) \Big(1+a\sin[\ln(1-|x|)]+a\cos[%
\ln(1-|x|)]\Big)\mathrm{d}x.   \label{def-niu}
\end{equation}
where $c$ is the normalizing constant. Then, the unique invariant measure is 
\begin{equation*}
\mathrm{d}\pi=\frac{\mathrm{d}\upsilon}{\theta(1-|x|)}=\frac{1}{2} \Big(1+a%
\sin[\ln(1-|x|)]+a\cos [\ln(1-|x|)]\Big)\mathrm{d}x. 
\end{equation*}
We first compute the following integral 
\begin{align*}
\int_{0}^{1-x}\frac{\mathrm{d}\pi}{1-y} & =\frac{1}{2}\int_{x}^{1}\frac{1}{t}%
\Big(1+a\sin[\ln(t)]+a\cos[\ln(t)]\Big)\mathrm{d}x \\
& =-\frac{1}{2}\ln x+O(1)\text{ as }x\rightarrow0,
\end{align*}
whence, by Theorem~\ref{thR(x)}%
\begin{equation*}
\lim_{n\rightarrow\infty}\frac{\mathrm{var}(S_{n})}{n\log n}=c. 
\end{equation*}
However, the covariances are not regularly varying because the spectral
measure is not. To see why it is enough to show that $r(x)$ is not regularly
varying at $0$. Indeed, if we take $y_{k}=e^{-2\pi k}\rightarrow0^{+},$ and $%
y_{k}=e^{\pi/2-2\pi k}\rightarrow0^{+},$ then $r(y_{k})=y_{k}$ and $%
r(y_{k})/y_{k}\rightarrow1.$ However, for the choice $r(z_{k})=z_{k}(1+2%
\alpha),$ we have $r(z_{k})/z_{k}\rightarrow1+2\alpha,$ and hence the
spectral measure is not regularly varying.
\end{example}

\begin{remark}
Often in the literature, conditions for the linear growth of the variance
are given in terms of the covariances (see for example \cite{HaggRos}). As
it turns out, one can construct positive covariance sequences such that $%
\sum_{k=0}^{n}\mathrm{cov}(X_{0},X_{k})=h(n)$ is slowly varying, and hence
the variance is regularly varying, but $a_{n}=\mathrm{cov}(X_{0},X_{n})>0,$
is not slowly varying. To construct such a chain, suppose that $\varepsilon
_{n}$ is an oscillating positive sequence such that $\varepsilon
_{n}\rightarrow 0$ and $\sum_{k}a_{k}=\infty $ where $a_{k}:=\varepsilon
_{k}/k$. Then $g_{n}=\sum_{k=1}^{n}a_{k}$ is slowly varying since 
\begin{equation*}
g_{[bn]}=g_{n}+\sum_{i=n+1}^{[bn]}\frac{\varepsilon _{i}}{i}=g_{n}+O\big(%
\varepsilon _{n+1}\log b\big).
\end{equation*}%
So, a priori we have many situations when $\mathrm{var}(S_{n})=nh(n)$ even
though the covariances (and hence the operator spectral measure) are not
regularly varying.
\end{remark}

\vskip10pt The above proof was direct in the sense, that it relied only on
the use of classical Tauberian theory without linking the transition
spectral measure with the spectral distribution function, and thus without
invoking the results of \cite{DU}. %
%
%

\subsection{Examples of limit theorems with non-linear normalizer}

As an application we construct a class of stationary irreducible and
aperiodic Markov Chains, based on the Metropolis-Hastings algorithm, with $%
\mathrm{var}(S_{n})\sim n h(n)$. Markov chains of this type are often
studied in the literature from different points of view, in Doukhan et al. 
\cite{Douk}, Rio (\cite{Rio1} and \cite{Rio2}), Merlev\`{e}de and Peligrad 
\cite{mp}, Zhao et al. \cite{ZWV}) and Longla et al. \cite{LPP}.

Let $E=\{|x|\leq1\}$ and define the transition kernel of a Markov chain by 
\begin{equation*}
Q(x,A)=|x|\delta_{x}(A)+(1-|x|)\upsilon(A), 
\end{equation*}
where $\delta_{x}$ denotes the Dirac measure and $\upsilon$ is a symmetric
probability measure on $[-1,1]$ in the sense that for any $A\subset
\lbrack0,1]$ we have $\upsilon(A)=\upsilon(-A).$ We shall assume that 
\begin{equation}  \label{eq:theta}
\theta=\int_{-1}^{1}\frac{1}{(1-|x|)}\upsilon(\mathrm{d}x)<\infty.
\end{equation}
We mention that $Q$ is a stationary transition function with the invariant
distribution 
\begin{equation*}
\mu(\mathrm{d}x)=\frac{1}{\theta(1-|x|)}\upsilon(\mathrm{d}x). 
\end{equation*}
Then, the stationary Markov chain $(\xi_{i})_{i}$ with values in $E$ and
transition probability $Q(x,A)$ and with marginal distribution $\mu$ is
reversible and positively recurrent. Moreover, for any odd function $g$ we
have $Q^{k}(g)(x)=|x|^{k}g(x)$ and therefore 
\begin{equation}
Q^{k}(g)(\xi_{0})=\mathbb{E}(g(\xi_{k})|\xi_{0})=|\xi_{0}|^{k}g(\xi_{0})%
\text{ a.s.}   \label{operator}
\end{equation}
For the odd function $g(x)=\mathrm{sgn}(x)$, define $X_{i}:=\mathrm{sgn}%
(\xi_{i})$. Then 
for any positive integer $k$ 
\begin{equation*}
\langle g,Q^{k}(g)\rangle={\int\nolimits_{-1}^{1}}|x|^{k}\mu(\mathrm{d}x)=2{%
\int\nolimits_{0}^{1}}x^{k}\mu(\mathrm{d}x), 
\end{equation*}
and so on $[0,1]$,\ $2\mu$ coincides with the transition spectral measure $%
\nu,$ associated to $Q$ and $g$. Furthermore, the operator $Q$ is of
positive type $\nu\lbrack-1,0)=0$. In other words 
\begin{equation*}
\nu= \left\{ 
\begin{array}{c}
2\mu\text{ on }[0,1] \\ 
0\text{ on }[-1,0)%
\end{array}
\right. . 
\end{equation*}
Therefore, by Theorem~\ref{thR(x)} applied with $\alpha=1$, $\mathrm{var}%
(S_{n})=nh(n)$ with $h(n)\rightarrow\infty$ slowly varying at $%
n\rightarrow\infty$ if and only if 
\begin{equation*}
V(x)=\int_{0}^{1-x}\frac{\nu(\mathrm{d}y)}{1-y} =\int_{0}^{1-x}\frac {2\mu(%
\mathrm{d}y)}{1-y}\sim\frac{1}{2} h\big(\frac{1}{x}\big), 
\end{equation*}
is slowly varying as $x\rightarrow0^{+}$.

Our next result presents a large class of transition spectral measures for
the model above which leads to functional central limit theorem.

\begin{theorem}
\label{thm:MH} Let $V(x)$ be slowly varying as $x\rightarrow0^{+}$. Then,
the central limit theorem, the functional central limit theorem and the
conditional central limit theorem hold for partial sums of $X_{i}$ defined
above.
\end{theorem}

%
\vskip10pt Next, we give a particular example of a Metropolis-Hastings
algorithm in which a non-degenerate central limit theorem holds under a
certain normalization. However, when normalized by the standard deviation we
have degenerate limiting distribution. \bigskip 

\noindent \textbf{Example.} For $0<x<1$, we take the slowly varying
function\ $V(x)=\exp (\sqrt{\ln (1/x)})$. By Theorem \ref{thR(x)}, as $%
n\rightarrow \infty $ 
\begin{equation*}
\mathrm{var}(S_{n})=2nV(1/n)(1+o(1)).
\end{equation*}%
%
On the other hand, let us choose $b_{n}$ such that $nE[\tau _{1}^{2}I(\tau
_{1}\leq b_{n})]/b_{n}^{2}\sim 1$ as $n\rightarrow \infty $. By Lemma~\ref%
{lem:aux} it follows that $2\theta nV(1/b_{n})\sim b_{n}^{2}$, with $\theta $
as defined in Eq.~\eqref{eq:theta}. Note now that 
\begin{align*}
2\theta nV(1/b_{n})& =2\theta n\exp (\sqrt{\ln (b_{n})})=2\theta n\exp (%
\sqrt{(1/2)\ln (4nV(1/b_{n})}) \\
& =2\theta n\exp (o(1)+\sqrt{(1/2)\ln n}+(1/2)^{3/2}),
\end{align*}%
which implies that $b_{n}^{2}\sim 2n\theta \exp (\sqrt{(1/2)\ln n})$, giving
the following CLT: 
\begin{equation*}
S_{n}/[2n\exp (\sqrt{(1/2)\ln n})^{1/2}\Rightarrow N(0,1).
\end{equation*}%
However 
\begin{equation*}
\frac{b_{n}^{2}}{nV(1/n)}\rightarrow 0\quad \text{and therefore\ }\frac{S_{n}%
}{\sqrt{\mathrm{var}(S_{n})}}\rightarrow ^{P}0.
\end{equation*}


\section{Continuous-time Markov processes}

Suppose we have a stationary Markov process $\{\xi_{t}\}_{t\geq0}$, with
values in the general state space $(S, \mathcal{A})$, and for $g \in
L_{0}^{2}(\pi)$ let $T_{t} g(x) := \mathbb{E} [ g(\xi_{t})| \xi_{0}=x]$.
Further $T_{t} = \mathrm{e}^{Lt}$, where $L$ is the infinitesimal generator
which we assume to be normal, which then implies that its spectrum is
supported on $\{z\in\mathbb{C}: \Re(z) \leq0\}$, such that 
\begin{equation*}
\mathrm{cov}(f(\xi_{t}), f(\xi_{0})) = \int_{\Re(z)\leq0} \mathrm{e}^{z t}
\nu(\mathrm{d} z). 
\end{equation*}
Finally define $S_{T}(g):= \int_{s=0}^{T} g(\xi_{s}) \mathrm{d} s$.

The following result is a continuous time analogue of Theorem~\ref%
{thm:harmonicmeasure}, linking the spectral distribution function with the
harmonic measure of planar Brownian motion.

\begin{theorem}
\label{pr:harmoniccts} 
Let $\{\xi_t\}_t$ be a stationary Markov process with normal generator $L$,
invariant measure $\pi$, and let $g\in L^2(\pi)$ and $\nu=\nu_f$ be the
transition spectral measure associated with $L$ and $g$. Write $%
(B_{t}^{z})_{t\geq0}$ for a standard planar Brownian motion in $\mathbb{C}$,
started at the point $z\in\mathbb{H}^{-}:= \{z\in\mathbb{C}: \Re(z)\leq 0\}$%
. Also let $Z$ be a random point in $\mathbb{H}^{-}$ distributed according
to $\nu$ and let $\tau_{\mathbb{H}^{-}}^{Z}:=\inf\{t {\geq} 0:B_{t}^{Z}\notin%
\mathbb{H}^{-}\}$. For $\alpha\in(0,2)$, the following statements are
equivalent:

\begin{description}
\item[\textit{(i)}] $\mathrm{var}(S_{T}(g))\sim T^{\alpha}h(T)$ as $%
n\rightarrow \infty$;

\item[\textit{(ii)}] $\mathrm{P}\big\{B_{\tau _{D}^{Z}}^{Z}\in (-\mathrm{i}x,%
\mathrm{i}x)\big\}\sim C(\alpha )x^{2-\alpha }h(1/x)/\nu (\mathbb{H}^{-})$
as $x\rightarrow 0^{+}$.
\end{description}
\end{theorem}

The following theorem gives a necessary and sufficient condition in terms of
the transition spectral measure $\nu$. Define for $x>0$, 
\begin{equation*}
U_{x} := \{ a+ \mathrm{i} b: 0\leq-a\leq|b|\leq x\}. 
\end{equation*}
and let 
\begin{equation*}
\varsigma^{2}:= -2\int_{\mathbb{H}^{-}}\Re(1/z)\nu(\mathrm{d}z).
\end{equation*}

\begin{theorem}
\label{pr:NSCcts}With the notation of Theorem~\ref{pr:harmoniccts} the
following are equivalent:

\begin{description}
\item[(i)] $\mathrm{var}(S_{T}(g))/T \to L = \varsigma^{2} +K$, where $K>0$;

\item[(ii)] $\varsigma ^{2}<\infty $ and $\nu (U_{x})/x\rightarrow K/\pi $
as $x\rightarrow 0^{+}$.
\end{description}

In addition, if $\varsigma^2<\infty$ and $\liminf_{T\to\infty}\mathrm{var}%
(S_{T})/T = \infty$, then $\mathrm{var}(S_{T})\sim T^{\alpha}h(T)$, for $%
\alpha\geq1$ and $h$ slowly varying at infinity, if and only if $%
\nu(U_{x})\sim C(\alpha) x^{2-\alpha }h(1/x)$.
\end{theorem}


\section{Proofs}

\noindent\textbf{Proof of Representation Lemma \ref{thm:normaldensity}.} For 
$t\in\lbrack-\pi,\pi]$ and $z\in D_{0}$ define the function 
\begin{equation*}
f(t,z):={\frac{1}{2\pi}}\Bigg[1+\sum_{k=1}^{\infty}\Big(z^{k}\mathrm{e}^{%
\mathrm{i}tk}+\bar{z}^{k}\mathrm{e}^{-\mathrm{i}tk}\Big)\Bigg]={\frac {1}{%
2\pi}}\frac{1-|z|^{2}}{|1-z\mathrm{e}^{\mathrm{i}t}|^{2}}, 
\end{equation*}
the Poisson kernel for the unit disk. Our approach is to integrate on $D_{0}$
with respect to $\nu_{0}(\mathrm{d}z)$, obtaining in this way a function
defined on $[0,2\pi]$ as follows 
\begin{equation*}
f(t)={\frac{1}{2\pi}}\int_{D_{0}}\left( 1+\sum_{k=1}^{\infty}\Big(z^{k}%
\mathrm{e}^{\mathrm{i}tk}+\bar{z}^{k}\mathrm{e}^{-\mathrm{i}tk}\Big)\right)
\nu_{0}(\mathrm{d}z)={\frac{1}{2\pi}}\int_{D_{0}}\frac{1-|z|^{2}}{|1-z%
\mathrm{e}^{\mathrm{i}t}|^{2}}\nu_{0}(\mathrm{d}z). 
\end{equation*}
The function is well defined since we are integrating the positive Poisson
kernel over the open disk, and in fact, by using polar coordinates, we also
have 
\begin{align*}
\int_{0}^{2\pi}f(t)\mathrm{d}t & ={\frac{1}{2\pi}}\int_{s=0}^{2\pi}\int
_{r=0}^{1^{-}}\int_{t=0}^{2\pi}\frac{1-r^{2}}{1-2r\cos(s+t)+r^{2}}\mathrm{d}%
t\,\nu_{0}(\mathrm{d}r,\mathrm{d}s) \\
& \leq{\frac{1}{2\pi}}\int_{s=0}^{2\pi}\int_{r=0}^{1^{-}}\frac{2\pi(1-r^{2})%
}{1-r^{2}}\,\nu_{0}(\mathrm{d}r,\mathrm{d}s)={2\pi}<\infty.
\end{align*}
Therefore, it is obvious that $f\in L^{1}(0,2\pi)$, and it makes sense to
calculate 
\begin{align}
\int_{0}^{2\pi}\mathrm{e}^{\mathrm{i}tn}f(t)\mathrm{d}t & ={\frac{1}{2\pi}}%
\int_{0}^{2\pi}\mathrm{e}^{\mathrm{i}tn}\int_{D_{0}}\Big[1+\sum_{k=1}^{%
\infty}\Big(z^{k}\mathrm{e}^{\mathrm{i}tk}+\bar{z}^{k}\mathrm{e}^{-\mathrm{i}%
tk}\Big)\Big]\nu_{0}(\mathrm{d}z)\mathrm{d}t  \label{comp} \\
& =\int_{D_{0}}z^{n}\nu_{0}(\mathrm{d}z).  \notag
\end{align}
Because of the decomposition 
\begin{equation*}
\mathrm{cov}(X_{0},X_{n})=\int_{D}z^{n}\mathrm{d}\nu(z)=\int_{D_{0}}z^{n}%
\mathrm{d}\nu_{0}(z)+\int_{0}^{2\pi}\mathrm{e}^{\mathrm{i}tn}\mathrm{\ }%
\nu_{\Gamma}(\mathrm{d}t), 
\end{equation*}
by (\ref{comp}) we obtain that%
\begin{equation*}
\mathrm{cov}(X_{0},X_{k})=\int_{0}^{2\pi}\mathrm{e}^{\mathrm{i}tn}f(t)%
\mathrm{d}t+\int_{0}^{2\pi}\mathrm{e}^{\mathrm{i}tn}\nu_{\Gamma }(\mathrm{d}%
t). 
\end{equation*}
Now, by (\ref{SpM}) the spectral distribution function $F$ associated with
the stationary sequence $(X_{i})_{i}$, is then given by (\ref{SMR}). \hfill$%
\square$

\bigskip

\bigskip\noindent\textbf{Proof of Corollary~\ref{prop:SpDens}.} The result
follows from Representation Lemma~\ref{thm:normaldensity} and Theorem A. To
obtain the last point in the theorem, by standard analysis, the spectral
measure has the following useful asymptotic representation 
\begin{align*}
F(x) & =\int_{0}^{x}f(t)\mathrm{d}t={\frac{1}{2\pi}}\int_{0}^{x}\int_{-1}^{1}%
\frac{1-\lambda^{2}}{1+\lambda^{2}-2\lambda\cos t}\nu(\mathrm{d}\lambda)dt \\
& =O(x)+{\frac{x}{\pi}}\int_{0}^{1}r(y)\frac{\mathrm{d}y}{x^{2}+y^{2}}\text{
\ \ as }x\rightarrow0^{+}.
\end{align*}
So, we derive $F(x)\sim {\frac{1}{2}}C(\alpha)x^{2-%
\alpha}h(1/x)$ if and only if 
\begin{equation*}
\int_{0}^{1}r(y)\frac{\mathrm{d}y}{x^{2}+y^{2}}\sim{\frac{\pi}{%
{2}} C(\alpha)}x^{1-\alpha }h(1/x)+O(1)\text{ \ \ as }%
x\rightarrow0^{+}. 
\end{equation*}

\bigskip

\noindent \textbf{Proof of Theorem~\ref{thm:harmonicmeasure}.} As usual, let 
$D$ be the closed unit disk, and $D_{0}$ its interior. From Representation
Lemma ~\ref{thm:normaldensity} the spectral density $f\in L_{1}([-\pi ,\pi ])
$ is given by the formula 
\begin{equation*}
f(t)={\frac{1}{2\pi }}\int_{D_{0}}\frac{1-|z|^{2}}{|1-z\mathrm{e}^{\mathrm{i}%
t}|^{2}}\nu (\mathrm{d}z).
\end{equation*}%
Notice that $D$ is regular for Brownian motion, in the sense that all points
in $\Gamma =\partial D$ are regular, i.e. for all $z\in \partial D$ and for $%
\tilde{\tau}_{D}^{z}:=\inf \{t>0:B_{t}^{z}\notin D\}$ we have $\mathrm{P}%
^{z}\{\tilde{\tau}_{D}^{z}=0\}=1$. The harmonic measure in $D$ from $z$ is
the probability measure on $\partial D$, $\mathrm{hm}(z,D;\cdot )$ given by 
\begin{equation*}
\mathrm{hm}(z,D;V)=\mathrm{P}^{z}\{B(\tau _{D}^{z})\in V\},
\end{equation*}%
where $\mathrm{P}^{z}$ denotes the probability measure of Brownian motion
started at the point $z$, and $V$ is any Borel subset of $\partial D$.

Since $\partial D$ is piecewise analytic, $\mathrm{hm}(z,D;\cdot )$ is
absolutely continuous with respect to Lebesgue measure (length) on $\partial
D$ and the density is the Poisson kernel (see for example \cite{Law}). In
the case of the unit disk $D$ the density for $\mathrm{hm}(z,D;\cdot )$ for $%
z\in D$, $w\in \partial D$ or $t\in \lbrack 0,2\pi ]$, is given by 
\begin{equation*}
H_{D}(z,w)=\frac{1}{2\pi }\frac{1-|z|^{2}}{|w-z|^{2}}=\frac{1}{2\pi }\frac{%
1-|z|^{2}}{|1-e^{it}z|^{2}}.
\end{equation*}%
Let $Z$ be a $D$-valued random variable with probability measure $\nu $
properly normalized, independent of the Brownian motion. 
Then 
\begin{align*}
\int_{-x}^{x}f(t)\mathrm{d}t& ={\frac{1}{2\pi }}\int_{D_{0}}\int_{-x}^{x}%
\frac{1-|z|^{2}}{|1-z\mathrm{e}^{\mathrm{i}t}|^{2}}\mathrm{d}t\,\nu _{0}(%
\mathrm{d}z) \\
& =\int_{D_{0}}\mathrm{P}^{z}\{B(\tau _{D}^{z})\in (-x,x)\}\nu _{0}(\mathrm{d%
}z).
\end{align*}%
On the other hand, since $\Gamma =\partial D$ is regular for Brownian
motion, we have for all $z\in \Gamma $, $\mathrm{P}^{z}(B(\tau _{D}^{z})\in
\Gamma _{x})=1$ if $z\in \Gamma _{x}$ and $0$ otherwise. Thus 
\begin{equation*}
\nu (\Gamma _{x})=\int_{\Gamma _{x}}\nu _{\Gamma }(\mathrm{d}z)=\int_{\Gamma
}\mathrm{P}^{z}(B(\tau _{D}^{z})\in \Gamma _{x})\nu _{\Gamma }(\mathrm{d}z).
\end{equation*}%
Therefore from Representation Lemma~\ref{thm:normaldensity} we have 
\begin{align*}
G(x):=\int_{-x}^{x}F(\mathrm{d}x)& =\int_{D_{0}}\mathrm{P}^{z}\{B(\tau
_{D}^{z})\in \Gamma _{x}\}\nu _{0}(\mathrm{d}z)+\int_{\Gamma }\mathrm{P}%
^{z}(B(\tau _{D}^{z})\in \Gamma _{x})\nu _{\Gamma }(\mathrm{d}z) \\
& =\nu (D)\mathrm{P}\Big(B^{Z}(\tau _{D}^{Z})\in \Gamma _{x}\Big)=:\nu
(D)H_{\nu ,D}(\Gamma _{x}).
\end{align*}%
The measure $H_{\nu ,D}(\cdot )$, is essentially the harmonic measure when
Brownian motion starts at a random point and stops when it hits $\partial D$%
. Finally, from Theorem A we conclude that $\mathrm{var}(S_{n})$ is
regularly varying if and only if the measure $H_{\nu ,D}$ is regularly
varying at the origin. \hfill $\square $

\bigskip

\noindent\textbf{Proof of Theorem~\ref{pr:NSC}.} The first part is motivated
by Gordin and Lif\v{s}ic \cite{GL} (see Theorem 7.1 in \cite{BI}; see also 
\cite{DL}). We write the martingale type orthogonal decomposition:%
\begin{equation*}
S_{n}=\mathbb{E}_{0}(S_{n})+\sum_{i=1}^{n}\mathbb{E}_{i}(S_{n}-S_{i-1})-%
\mathbb{E}_{i-1}(S_{n}-S_{i-1}). 
\end{equation*}
So 
\begin{align*}
\mathrm{var}(S_{n}) & =\mathbb{E}(\mathbb{E}_{0}(S_{n}))^{2}+\sum_{i=1}^{n}%
\mathbb{E(E}_{i}(S_{n}-S_{i-1})-\mathbb{E}_{i-1}(S_{n}-S_{i-1}))^{2} \\
& =\mathbb{E}(\mathbb{E}_{0}(S_{n}))^{2}+\sum_{i=1}^{n}\mathbb{E(E}%
_{1}(S_{i})-\mathbb{E}_{0}(S_{i}))^{2}.
\end{align*}
By applying spectral calculus, 
\begin{align*}
& =\mathbb{E}(\mathbb{E}_{0}(S_{n}))^{2}+\sum_{j=1}^{n}%
\int_{D}|1+z+...+z^{j-1}|^{2}(1-|z|^{2})\nu(\mathrm{d}z) \\
& =\mathbb{E}(\mathbb{E}_{0}(S_{n}))^{2}+\sum_{j=1}^{n}\int_{D_{0}}\frac{%
|1-z^{j}|^{2}(1-|z|^{2})}{|1-z|^{2}}\nu(\mathrm{d}z) \\
& =\mathbb{E}(\mathbb{E}_{0}(S_{n}))^{2}+n\int_{D_{0}}\delta_{n}(z)\frac{%
(1-|z|^{2})}{|1-z|^{2}}\nu(\mathrm{d}z).
\end{align*}
Note that $\delta_{n}(z)\leq4$ and for all $z\in D_{0}$ 
\begin{equation*}
\delta_{n}(z)=\frac{1}{n}\sum_{j=1}^{n}|1-z^{j}|^{2}\rightarrow1\text{ as }%
n\rightarrow\infty. 
\end{equation*}
Thus, by the Lebesgue dominated theorem and our conditions 
\begin{equation*}
\lim_{n\rightarrow\infty}\int_{D_{0}}\delta_{n}(z)\frac{(1-|z|^{2})}{%
|1-z|^{2}}\nu(\mathrm{d}z)=\int_{D_{0}}\frac{1-|z|^{2}}{|1-z|^{2}}\nu(%
\mathrm{d}z)=\sigma^{2}. 
\end{equation*}
This, along with Fatou's lemma, proves that $\mathrm{var}(S_{n})/n%
\rightarrow K$ exists if and only if $\sigma^{2}<\infty$ and 
\begin{equation*}
\frac{1}{n}\mathbb{E}(\mathbb{E}_{0}(S_{n}))^{2}\rightarrow L=K-\sigma^{2}. 
\end{equation*}
Now, let us introduce a new measure on $D_{0}$ 
\begin{equation*}
\mu(z)=\frac{1-|z|}{|1-z|^{2}}\nu(\mathrm{d}z), 
\end{equation*}
which is finite when $\sigma^{2}<\infty$.

To complete the proof of the first part of the theorem, we notice that by
the spectral calculus 
\begin{align*}
\mathbb{E}(\mathbb{E}_{0}(S_{n}))^{2} & =\int_{D}|z+...+z^{n}|^{2}\nu(%
\mathrm{d}z)\;=\int_{D}|z|^{2}\frac{|1-z^{n}|^{2}}{|1-z|^{2}}\nu(\mathrm{d}z)
\\
& =\int_{D}\frac{|1-z^{n}|^{2}}{|1-z|^{2}}\nu(\mathrm{d}z)-%
\int_{D}|1-z^{n}|^{2}(1+|z|)\mu(\mathrm{d}z) \\
& =\int_{D}\frac{|1-z^{n}|^{2}}{|1-z|^{2}}\nu(\mathrm{d}z)+O(1),
\end{align*}
since $\mu$ is a finite measure.

To prove the second part of this theorem, we show equivalence of (i) and
(ii) and then of (ii) and (iii). In addition, throughout we use notation: 
\begin{equation*}
z=|z|e^{\mathrm{i}\mathrm{Arg}(z)},|z|=1-y,\theta=\mathrm{Arg}(z). 
\end{equation*}

We note first that 
\begin{equation*}
|1-z^{n}|^{2}=(1-|z|^{n})^{2}+|z|^{n}\sin^{2}(n\theta/2). 
\end{equation*}
The proof strategy consists in showing, several successive approximation
steps, that 
\begin{equation*}
\frac{1}{n}\mathbb{E}(\mathbb{E}_{0}(S_{n}))^{2}  = \int_0^{\pi} \frac{%
\sin^2(n\theta/2)}{\sin^2 (\theta/2)} G(\mathrm{d} \theta) + o(1), 
\end{equation*}
for some appropriate measure $G$, and then to apply Theorem A. With this in
mind we write 
\begin{align*}
I_{n} & =\frac{1}{n}\int_{D}\frac{|1-z^{n}|^{2}}{|1-z|^{2}}\nu(\mathrm{d}z)
\\
& =\frac{1}{n}\int_{D}\frac{|z|^{n}\sin^{2}(n\mathrm{Arg}(z)/2)}{|1-z|^{2}}%
\nu(\mathrm{d}z)+\frac{1}{n}\int_{D}\frac{(1-|z|^{n})^2}{|1-z|^{2}}\nu(%
\mathrm{d}z)=:I_{n}^{\prime}+\Delta_{n}^{\prime}.
\end{align*}
Note that 
\begin{align*}
\Delta_{n}^{\prime} & =\int_{D_{0}}\Big(\frac{(1-|z|^{n})}{n(1-|z|)}\Big)%
(1-|z|^{n})\frac{(1-|z|)}{|1-z|^{2}}\nu(\mathrm{d}z)\; \\
& =\int_{D}\frac{1}{n}\Big(\sum_{j=0}^{n-1}|z|^{j}\Big)(1-|z|^{n})\mu(%
\mathrm{d}z).
\end{align*}
By Lebesgue dominated convergence theorem, since the bounded integral
argument goes to $0$ for each $|z|\leq1,$ we have $\Delta_{n}^{\prime}%
\rightarrow0$ as $n\rightarrow\infty$.

Then, write 
\begin{align*}
I_{n}^{\prime} & =\frac{1}{n}\int_{D}\frac{\sin^{2}(n\mathrm{Arg}(z)/2)}{%
|1-z|^{2}}\nu(\mathrm{d}z)\;-\;\frac{1}{n}\int_{D}(1-|z|^{n})\frac{\sin^{2}(n%
\mathrm{Arg}(z)/2)}{|1-z|^{2}}\nu(\mathrm{d}z) \\
& =:I_{n}^{\prime\prime}+\Delta_{n}^{\prime\prime},
\end{align*}
and again 
\begin{align*}
\Delta_{n}^{\prime\prime} & =\int_{D_{0}}\Big(\frac{(1-|z|^{n})}{n(1-|z|)}%
\Big)\sin^{2}(n\mathrm{Arg}(z)/2)\frac{(1-|z|)}{|1-z|^{2}}\nu(\mathrm{d}z)\;
\\
& =\int_{D_{0}}\Big(\frac{(1-|z|^{n})}{n(1-|z|)}\Big)\sin^{2}(n\mathrm{Arg}%
(z)/2)\mu(\mathrm{d}z).
\end{align*}
Note that by Lebesgue dominated theorem $\Delta_{n}^{\prime\prime}%
\rightarrow0$ as $n\rightarrow\infty$, since the bounded integral argument
goes to $0$ for each $|z|<1$.

Fix now a small positive $a>0$, recall that $z=(1-y)e^{i\theta }$ and define
an auxiliary subset of $D$ 
\begin{equation*}
D_{a}=\{z=(1-y)e^{i\theta }\;:\;0<|\theta |\leq a\;,\;0\leq y\leq a\;\}.
\end{equation*}%
Further, notice that by the dominated convergence theorem 
\begin{equation*}
\varepsilon _{n}=\int_{D_{a}}\Big|\frac{\sin (n\mathrm{Arg}(z)/2)}{n\mathrm{%
Arg}(z)}\Big|\mu (\mathrm{d}z)\rightarrow 0\text{ as }n\rightarrow \infty ,
\end{equation*}%
since the bounded integral argument goes to $0$ for each $|z|<1$.\newline
Let $N$ be large enough, so that $|\epsilon _{n}|<1$ for all $n\geq N$, and
take {$\delta_{n}=\max(\mathrm{e}^{-n},\sqrt{%
\varepsilon_{n}})$} 
so that $\delta
_{n}>0$ for all $n$. In this way $\epsilon _{n}/\delta _{n}=0$ is
well-defined if $\epsilon _{n}=0$, and $\epsilon _{n}/\delta _{n}\leq \sqrt{%
\epsilon _{n}}\rightarrow 0$ as $n\rightarrow \infty $. Further define two
auxiliary sequences of subsets of $D_{a}$ 
\begin{align*}
D_{a,n}& =\{z=(1-y)e^{i\theta }\;:\;0<\delta _{n}|\theta |\leq y\leq
a,|\theta |\leq a\;\}\;, \\
U_{a,n}& =U_{a}\setminus {D_{a,n}}=\{z=(1-y)e^{i\theta }\;:\;0\leq y<\delta
_{n}|\theta |,\;|\theta |\leq a\;\}.
\end{align*}%
Next, let 
\begin{equation*}
g_{n}(z)=\frac{\sin ^{2}(n\mathrm{Arg}(z)/2)}{n|1-z|^{2}}
\end{equation*}%
and write 
\begin{equation*}
I_{n}^{\prime \prime }=\frac{1}{n}\int_{D}g_{n}(z)\nu (\mathrm{d}z)=\frac{%
O(1)}{n}+\int_{D_{a}}g_{n}(z)\nu (\mathrm{d}z)
\end{equation*}%
and 
\begin{equation*}
\int_{D_{a}}g_{n}(z)\nu (\mathrm{d}z)=\int_{D_{a,n}}g_{n}(z)\nu (\mathrm{d}%
z)+\int_{U_{a,n}}g_{n}(z)\nu (\mathrm{d}z)=:\Delta _{n}^{\prime \prime
\prime }+I_{n}^{\prime \prime \prime }\text{ .}
\end{equation*}%
Notice that by construction ($\theta =\mathrm{Arg}(z)$) 
\begin{align*}
\Delta _{n}^{\prime \prime \prime }& \leq \int_{D_{a,n}}\Big|\frac{\sin
(n\theta /2)}{n\theta }\Big||\sin (n\theta /2)|\frac{|\theta |}{|1-z|^{2}}%
\nu (\mathrm{d}z)\; \\
& \leq \frac{1}{\delta _{n}}\int_{D_{a,n}}\Big|\frac{\sin (n\theta /2)}{%
n\theta }\Big|\frac{1-|z|}{|1-z|^{2}}\nu (\mathrm{d}z)\; \\
& =\frac{1}{\delta _{n}}\int_{D_{a,n}}\Big|\frac{\sin (n\theta /2)}{n\theta }%
\Big|\mu (\mathrm{d}z)\leq \frac{1}{\delta _{n}}\int_{D_{a}}\Big|\frac{\sin
(n\theta /2)}{n\theta }\Big|\mu (\mathrm{d}z)\;=\frac{\varepsilon _{n}}{%
\delta _{n}}\rightarrow 0.
\end{align*}%
In addition, on $U_{a,n}$ 
\begin{equation*}
|1-z|^{2}=y^{2}+(1-y)\sin ^{2}(\theta /2)=\sin ^{2}(\theta /2)(1+\delta
_{z,n}),
\end{equation*}%
where $|\delta _{z,n}|\leq c\delta _{n}$ for some positive $c$ and hence 
\begin{align*}
I_{n}^{\prime \prime \prime }& =\frac{1}{n}\int_{U_{a,n}}\frac{\sin
^{2}(n\theta /2)}{|1-z|^{2}}\nu (\mathrm{d}z) \\
& =\frac{1}{n}\int_{U_{a,n}}\frac{\sin ^{2}(n\theta /2)}{\sin ^{2}(\theta /2)%
}\nu (\mathrm{d}z)\;+\;\frac{1}{n}\int_{U_{a,n}}\frac{\sin ^{2}(n\theta /2)}{%
\sin ^{2}(\theta /2)}\Big(1-\frac{1}{1+\delta _{z,n}}\Big)\nu (\mathrm{d}z)
\\
& =:I_{n}^{(4)}+\Delta _{n}^{(4)}.
\end{align*}%
By construction $\Delta _{n}^{(4)}=o(1)I_{n}^{iv}$ and then $I_{n}^{\prime
\prime \prime }=I_{n}^{(4)}(1+o(1))$. Write 
\begin{align*}
I_{n}^{(4)}& =\frac{1}{n}\int_{U_{a}}\frac{\sin ^{2}(n\theta /2)}{\sin
^{2}(\theta /2)}\nu (\mathrm{d}z)\;-\;\frac{1}{n}\int_{U_{a}\cap D_{a,n}}%
\frac{\sin ^{2}(n\theta /2)}{\sin ^{2}(\theta /2)}\nu (\mathrm{d}z). \\
& =I_{n}^{(5)}-\Delta _{n}^{(5)}.
\end{align*}%
In a similar way as $\Delta _{n}^{\prime \prime \prime }$ has been estimated
in Step 3, by construction ($\theta =\mathrm{Arg}(z)$) we obtain 
\begin{align*}
\Delta _{n}^{(5)}& \leq \int_{U_{a}\cap D_{a,n}}\Big|\frac{\sin (n\theta /2)%
}{n\theta }\Big||\sin (n\theta /2)||(1+\delta _{z,n})|\frac{|\theta |}{%
|1-z|^{2}}\nu (\mathrm{d}z)\; \\
& \leq \frac{1+c\delta _{n}}{\delta _{n}}\int_{D_{a,n}}\Big|\frac{\sin
(n\theta /2)}{n\theta }\Big|\frac{1-|z|}{|1-z|^{2}}\nu (\mathrm{d}z)\; \\
& \leq \frac{1+c\delta _{n}}{\delta _{n}}\int_{D_{a}}\Big|\frac{\sin
(n\theta /2)}{n\theta }\Big|\mu (\mathrm{d}z)=\frac{(1+c\delta
_{n})\varepsilon _{n}}{\delta _{n}}\rightarrow 0\text{ as }n\rightarrow
\infty .
\end{align*}%
Finally, let us define the set 
\begin{equation*}
D_{+}=\{z=(1-y)e^{i\theta }\;:\;0\leq y\leq |\theta |<\pi \}\;.
\end{equation*}%
Then, it follows that $I_{n}/n\rightarrow L$ if and only if 
\begin{equation*}
J_{n}=\int_{D_{+}}I_{n}(\mathrm{Arg}(z))\nu (\mathrm{d}z)\rightarrow 0,
\end{equation*}%
where $I_{n}$ is the Fejer kernel, 
\begin{equation*}
I_{n}(x):=\frac{\sin ^{2}(nx/2)}{n\sin ^{2}(x/2)}.
\end{equation*}%
Define $G(x)=\nu (U_{x})$ and notice that it is a non-negative,
non-decreasing bounded function. In addition, for any step function $%
g(\theta )=I(u<|\theta |\leq v)=I(|\theta |\leq v)-I(|\theta |\leq w)$ with $%
u,w$ being continuity points of $G(x)$ we have 
\begin{align*}
\int_{D_{+}}g(\mathrm{Arg}(z))\nu (\mathrm{d}z)& =\int_{D_{+}}I(g(\theta
))\nu (\mathrm{d}z)=\int_{D_{+}}I(|\theta |\leq v)\nu (\mathrm{d}%
z)-\int_{D_{+}}I(|\theta |\leq w)\nu (\mathrm{d}z) \\
& =\nu (U_{v})-\nu (U_{w})=\int g(x)dG(x)
\end{align*}%
and then, by Caratheodory and Lebesgue theorem we have 
\begin{equation*}
J_{n}=\int_{0}^{\pi }I_{n}(x)dG(x).
\end{equation*}%
By Theorem A, \ $J_{n}/n\rightarrow L$ if and only if $L={\pi
C}$ where $G(x)/x=\nu (U_{x})/x\rightarrow C$ as $x\rightarrow 0^{+}$,
completing the proof of equivalence of (i) and (ii) in Part (b).

To prove the equivalence of (ii) and (iii) in Part (b) under the finiteness
of integral \eqref{normalsigma}, let us define 
\begin{align*}
W_{x}& =\{z=(1-r)e^{iu}\in D\;:\;0\leq |u|\leq r\leq x\}\;, \\
D_{1/x}& =\{z=(1-r)e^{iu}\in D\;:\;0\leq |u|,r\leq x\}\;.
\end{align*}%
Since $\nu ({1})=0$, it follows that $\nu (W_{x})\rightarrow 0$ as $%
x\rightarrow 0^{+}$.

On the other hand, on $W_{x}$ 
\begin{equation*}
\frac{1-|z|}{|1-z|^{2}}\geq \frac{1}{2x}
\end{equation*}%
and hence by \eqref{normalsigma}, 
\begin{equation*}
0\leftarrow \int_{W_{x}}\frac{1-|z|}{|1-z|^{2}}\nu (\mathrm{d}z)\geq \frac{1%
}{2x}\nu (W_{x}),
\end{equation*}%
which implies that $\nu (D_{1/x})/x\rightarrow C$ if and only if (ii) holds
as $x\rightarrow 0^{+}$ and this completes the proof. \hfill $\square $

\begin{remark}
The sufficient part in Theorem \ref{pr:NSC} can be derived directly by
performing coordinate transformation mapping the open unit disk to the upper
half-plane and further careful analysis. Briefly, we change coordinates to
the upper half-plane $\mathbb{H}:=\{(a,b):a\in\mathbb{R},b>0\}$, via the
inverse Cayley transform $\phi(w):\mathbb{H}\rightarrow D_{0}$, where $\phi
(w):=(1+\mathrm{i}w)/(1-\mathrm{i}w)$. The finite measure $\nu$ is
transformed to a finite measure $\rho$ on $\mathbb{H}$, which for simplicity
we can assume to be a probability measure. Then, for $w:=a+\mathrm{i}b$ we
have 
\begin{equation*}
\sigma^{2}=\int_{\mathbb{H}}\frac{1-|\phi(a+\mathrm{i}b)|^{2}}{|1-\phi (a+%
\mathrm{i}b)|^{2}}\mathrm{d}v(\phi^{-1}(z))=\int_{a=-\infty}^{\infty}%
\int_{b=0}^{\infty}\frac{b}{a^{2}+b^{2}}\mathrm{d}\rho(a,b). 
\end{equation*}
and a further change of variables $z=\tan(t/2)$ gives 
\begin{equation*}
\int_{t=0}^{x}f(t)\mathrm{d}t=\frac{1}{2\pi}\int_{t=0}^{x}\iint_{\mathbb{H}}%
\frac{b}{(a+t/2)^{2}+b^{2}}\rho(\mathrm{d}a,\mathrm{d}b)\mathrm{d}t+o(x). 
\end{equation*}
\end{remark}

\bigskip

\noindent \textbf{Proof of Theorem \ref{thR(x)}.} Let $1\leq \alpha \leq 2$
and denote $C(n):=\sum_{i=0}^{n-1}\mathrm{cov}(X_{0},X_{i}).$ We start by
the well known representation 
\begin{equation*}
\mathrm{var}(S_{n})=n[\frac{2}{n}\sum_{k=1}^{n}C(k)-E(X_{0}^{2})].
\end{equation*}%
It is clear then, since $\mathrm{var}(S_{n})/n\rightarrow \infty $, that $%
\mathrm{var}(S_{n})$ has the same asymptotic behavior as $%
2\sum_{k=1}^{n}C(k).$ Implementing the notations%
\begin{equation*}
a_{k}=\int\nolimits_{0}^{1}x^{j}\nu _{1}(\mathrm{d}x)\text{ and }%
b_{k}=\int\nolimits_{-1}^{0}x^{j}\nu (\mathrm{d}x),
\end{equation*}%
where $\nu _{1}$ coincides with $\nu $ on $(0,1]$, and $\nu _{1}(\{0\})=0,$
we have the representation 
\begin{equation*}
C(k)=\sum_{j=0}^{k-1}a_{j}+\sum_{j=0}^{k-1}b_{j}=C_{1}(k)+C_{2}(k).
\end{equation*}%
We shall show that the terms $C_{1}(k)$ have a dominant contribution to the
variance of partial sum. To analyze $C_{2}(k)$ it is convenient to make
blocks of size $2$, a trick that has also appeared in \cite{HaggRos} where
it is attributed to \cite{Geyer}. We notice that 
\begin{equation*}
c_{l}=b_{2l}+b_{2l+1}=\int\nolimits_{-1}^{0}(x^{2l}+x^{2l+1})d\nu >0.
\end{equation*}%
Furthermore, for all $m$%
\begin{equation*}
\sum_{l=0}^{m-1}c_{l}=\sum_{l=0}^{m-1}\int\nolimits_{-1}^{0}x^{2l}(1+x)\nu (%
\mathrm{d}x)=\int\nolimits_{-1}^{0}\frac{1-x^{2m}}{1-x}\nu (\mathrm{d}%
x)\leq E(X_{0}^{2}).
\end{equation*}%
Therefore $|C_{2}(k)|\leq 2E(X_{0}^{2})$ and so, $\sum_{k=1}^{n}C_{2}(k)\leq
2nE(X_{0}^{2})$. Because $\mathrm{var}(S_{n})/n\rightarrow \infty $ we note
that $\mathrm{var}(S_{n})$ has the same asymptotic behavior as $%
2\sum_{k=1}^{n}C_{1}(k).$ Now, each $C_{1}(k)=\sum_{j=0}^{k-1}a_{k}$ with $%
a_{k}>0.$ So, because the sequence $(C_{1}(k))_{k}$ is increasing, by the
monotone Tauberian theorem (see Theorem \ref{thm:tm} in {\cite[Cor 1.7.3]%
{Bing}}) for all $\alpha \geq 1\mathrm{\ }$we have

\begin{equation}
\mathrm{var}(S_{n})\sim 2\sum_{k=1}^{n}C_{1}(k)\sim n^{\alpha }h(n)\quad 
\text{ if and only if }\quad C_{1}(n)\sim \alpha n^{\alpha -1}h(n)/2.\ 
\label{rel1}
\end{equation}%
Note now that 
\begin{equation*}
C_{1}(n)=\int\nolimits_{0}^{1}\frac{1-x^{n}}{1-x}\nu _{1}(\mathrm{d}x).
\end{equation*}%
It is convenient to consider the transformation $T:[0,1]\rightarrow \lbrack
0,1]$ defined by $T(x)=1-x$. For a Borelian $A$ of $[0,1]$ define the measure%
\begin{equation}
r(A)=\nu _{1}(T(A)).  \label{def r}
\end{equation}%
Then 
\begin{equation*}
C_{1}(n)=\int\nolimits_{0}^{1}\frac{1-(1-y)^{n}}{y}\ r(\mathrm{d}y).
\end{equation*}%
We shall integrate by parts. Denote%
\begin{equation*}
R(u)=\int\nolimits_{u}^{1}\frac{1}{y}r(\mathrm{d}y)\text{ and }%
U_{n}(u)=[1-(1-u)^{n}].
\end{equation*}%
Let $0<b<1.$ By the definition of $\nu _{1}$ we have $r(\{1\})=0.$ Since $%
U_{n}$ is continuous, by \cite[Theorem~18.4]{Bill95},%
\begin{equation*}
\int\nolimits_{b}^{1}\frac{1-(1-y)^{n}}{y}r(\mathrm{d}%
y)=[1-(1-b)^{n}]R(b)+n\int\nolimits_{b}^{1}(1-u)^{n-1}R(u)\mathrm{d}u.
\end{equation*}%
Note that 
\begin{equation*}
\lim \sup_{b\rightarrow 0^{+}}[1-(1-b)^{n}]R(b)\leq nbR(b)\leq n\mathbb{E}%
(X_{0}^{2})=o(\sigma _{n}^{2}).
\end{equation*}%
Therefore%
\begin{equation*}
C_{1}(n)=o(\sigma _{n}^{2})+n\int\nolimits_{0}^{1}(1-u)^{n-1}R(u)\mathrm{d}%
u.
\end{equation*}%
By the change of variables $1-u=\mathrm{e}^{-y}$ we have 
\begin{equation*}
\int\nolimits_{0}^{1}(1-u)^{n-1}R(u)du=\int\nolimits_{0}^{\infty }R(1-%
\mathrm{e}^{-y})\mathrm{e}^{-yn}\mathrm{d}y.
\end{equation*}%
It follows that%
\begin{gather}
C_{1}(n)\sim \alpha n^{\alpha -1}h(n)/2\quad \text{ if and only if}\quad
n\int\nolimits_{0}^{\infty }R(1-\mathrm{e}^{-y})\mathrm{e}^{-yn}dx\sim
\alpha n^{\alpha -1}h(n)/2  \label{rel2} \\
\int\nolimits_{0}^{\infty }\quad \text{ if and only if}\quad {\displaystyle%
\int\nolimits_{0}^{\infty }}R(1-\mathrm{e}^{-y})\mathrm{e}^{-yn}\mathrm{d}%
x\sim \alpha n^{\alpha -2}h(n)/2\text{.}  \notag
\end{gather}%
From here we shall apply Karamata's Tauberian Theorem~\ref{t1}. Since $%
\alpha \leq 2,$%
\begin{gather}
\int\nolimits_{0}^{\infty }R(1-\mathrm{e}^{-y})\mathrm{e}^{-yn}\mathrm{d}%
x\sim \alpha n^{\alpha -2}h(n)/2\text{ as }n\rightarrow \infty 
\label{rel 3} \\
\text{if and only if }\,\,\int\nolimits_{0}^{x}R(1-\mathrm{e}^{-y})\mathrm{d}%
y\sim c_{\alpha }^{\prime }x^{2-\alpha }h(1/x)\text{ as }x\rightarrow 0^{+},
\notag
\end{gather}%
where $c_{\alpha }^{\prime }=\alpha /[2\Gamma (3-\alpha )].$ Again by the
monotone Karamata Theorem~\ref{thm:tm} this happens if and only if 
\begin{equation}
R(1-\mathrm{e}^{-y})\sim (2-\alpha )c_{\alpha }^{\prime }x^{1-\alpha }h(1/x)%
\text{ as }x\rightarrow 0^{+}.  \label{rel4}
\end{equation}%
Changing variables $x=1-\mathrm{e}^{-y}$ and taking into account Karamata's
representation for slowly varying functions we get%
\begin{align}
R(x)& \sim c_{\alpha }\Big[{\ln (1/(1-x))}\Big]^{{1}-\alpha }h({-1/\ln (1-x)}%
)\   \label{rel 5} \\
& \sim c_{\alpha }x^{{1}-\alpha }h(1/x)\ \text{as }x\rightarrow 0^{+}, 
\notag
\end{align}%
where $c_{\alpha }=(2-\alpha )c_{\alpha }^{\prime }$. By combining the
results in relations (\ref{rel1})-(\ref{rel 5}) we have that%
\begin{equation*}
\mathrm{var}(S_{n})\sim n^{\alpha }h(n)\quad \text{ if and only if }\quad
R(x)\sim c_{\alpha }x^{{1}-\alpha }h(1/x)\ \text{as }x\rightarrow 0^{+}.
\end{equation*}%
It remains to note that 
\begin{equation*}
R(x)=\int\nolimits_{x}^{1}\frac{1}{y}r(\mathrm{d}y)=\int\nolimits_{0}^{1-x}%
\frac{1}{1-y}\upsilon _{1}(\mathrm{d}y)\sim V(x)\text{ as }x\rightarrow
0^{+}.
\end{equation*}

When $1<\alpha<2,$ one can say more: the distribution function induced by
the spectral measure is regularly varying.

Note that again by Theorem~\ref{thm:tm}, since the sequence $a_{k}$ is a
monotone sequence of positive numbers and $\alpha-1>0$ we have 
\begin{equation*}
C_{1}(n)=\sum_{k=0}^{n-1}a_{k}=\alpha n^{\alpha-1}h(n)/2\quad\text{ if and
only if }\quad a_{n}\sim\alpha(\alpha-1)n^{\alpha-2}h(n)/2. 
\end{equation*}
Now, by considering the mapping $T^{\prime}:[0,1]\rightarrow\lbrack0,\infty),
$ given by $T^{\prime -x},$ we obtain 
\begin{equation*}
a_{n}=\int_{0}^{1}t^{n}\ \nu_{1}(\mathrm{d}t)=\int_{0}^{\infty}e^{-nx}\psi(%
\mathrm{d}x), 
\end{equation*}
where $\psi(A)=r(T^{\prime}(A)),$ for $A$ Borelian in $[0,\infty)$. Letting $%
d_{\alpha}:=\alpha(\alpha-1)/2\Gamma(3-\alpha),$ it follows by Theorem \ref%
{t1} ({\cite[Thm~1.7.1']{Bing}}), that%
\begin{equation*}
a_{n}\sim\alpha(\alpha-1)n^{\alpha-2}h(n)/2\text{ as }n\rightarrow\infty\quad%
\text{ iff }\quad\psi\lbrack0,x]\sim d_{\alpha}x^{2-\alpha}h(1/x)\text{ as }%
x\rightarrow0^{+}. 
\end{equation*}
Then, we obtain as before, by the properties of slowly varying functions 
\begin{equation*}
\psi\lbrack0,x]\sim d_{\alpha}x^{2-\alpha}h(1/x)\text{ as }x\rightarrow 0^{+}%
\text{ iff }\nu_{1}(1-x,1]\sim d_{\alpha}x^{2-\alpha}h(1/x)\text{ as }%
x\rightarrow0^{+}. 
\end{equation*}
This last relation combined with (\ref{rel2}) gives the last part of the
theorem.\hfill$\square$

\bigskip

\begin{proof}[Proof of Theorem~\ref{thm:MH}] We only prove the central
limit theorem, with the other results following in a similar manner. Our
approach is based on the regeneration process. Define
\begin{align*}
T_{0} &  =\inf\{i>0:\xi_{i}\neq\xi_{0}\}\\
T_{k+1} &  =\inf\{i>T_{k}:\xi_{i}\neq\xi_{i-1}\},
\end{align*}
and let $\tau_{k}:=T_{k+1}-T_{k}$. It is well known that $(\xi_{\tau_{k}}%
,\tau_{k})_{k\geq1}$ are i.i.d. random variables with $\xi_{\tau_{k}}$ having
the distribution $\upsilon.$ Furthermore,
\[
\mathbb{P}(\tau_{1}>n|\xi_{\tau_{1}}=x)=|x|^{n}\text{.}%
\]
Then, it follows that%
\[
\mathbb{E}(\tau_{1}|\xi_{\tau_{1}}=x)=\frac{1}{1-|x|}\text{ }\ \ \text{and
\ \ }\mathbb{E}(\tau_{1})=%
{\displaystyle\int\nolimits_{-1}^{1}}
\frac{1}{1-|x|}\upsilon(\mathrm{d}x)=\theta.
\]
So, by the law of large numbers $T_{n}/n\rightarrow\theta$ a.s. Let us study
the tail distribution of $\tau_{1}.$ Since
\[
\mathbb{P}(\tau_{1}|X_{\tau_{1}}|>u|\xi_{\tau_{1}}=x)=\mathbb{P}(\tau
_{1}>u|\xi_{\tau_{1}}=x)=|x|^{u},
\]
by integration we obtain%
\begin{equation}
\mathbb{P}(\tau_{1}>u)=\int_{-1}^{1}|x|^{u}\upsilon(\mathrm{d}x)=2\int_{0}%
^{1}x^{u}\upsilon(\mathrm{d}x).\label{tail tau}%
\end{equation}
Using now the relation between $\upsilon(\mathrm{d}x)$ and $\mu(\mathrm{d}x)$
and symmetry we get
\[
\mathbb{P}(\tau_{1}>u)=2\theta\int_{0}^{1}x^{u}(1-x)\mu(\mathrm{d}%
x)=\theta\int_{0}^{1}x^{u}(1-x)\nu(\mathrm{d}x),
\]
where $\nu$ is spectral measure. By using the fact that $V(x)$ is slowly
varying and Lemma~\ref{lem:aux} in Section "Technicalities" it follows that
\begin{equation}
H(u):=E[\tau_{1}^{2}I(\tau_{1}\leq u)]\text{ is slowly varying as
}u\rightarrow\infty.\label{Feller}%
\end{equation}
For each $n$, let $m_{n}$ be such that $T_{m_{n}}\leq n<T_{m_{n}+1}.$ Note
that we have the representation
\begin{equation}
\sum_{k=1}^{n}X_{k}-\sum_{k=1}^{[n/\theta]}Y_{k}=(T_{0}-1)X_{0}+(\sum
_{k=1}^{m_{n}}\tau_{k}X_{\tau_{k}}-\sum_{k=1}^{[n/\theta]}\tau_{k}X_{\tau_{k}%
})+\sum_{k=T_{m_{n}+1}}^{n}X_{k}\text{,}\label{Rep}%
\end{equation}
\hfill where $Y_{k}=\tau_{k}X_{\tau_{k}}$ is a centered i.i.d. sequence, and
by (\ref{Feller}) is in the domain of attraction of a normal law (see
Feller~\cite{Feller66}). Therefore,
\begin{equation}
\frac{\sum_{k=1}^{[n/\theta]}Y_{k}}{b_{[n/\theta]}}\Rightarrow N(0,1)\text{.}%
\label{Fell}%
\end{equation}
where $b_{n}^{2}\sim nH(b_{n}).$ The rest of the proof is completed on the
same lines as in the proof of Example 12 in \cite{LPP}, the final result being
that%
\begin{equation*}
\frac{\sum_{k=1}^{n}X_{k}}{b_{[n/\theta]}}\Rightarrow N(0,1)\text{.}%
\end{equation*}
\end{proof}

\begin{proof}[Proof of Theorem~\ref{pr:harmoniccts}]
The proof is similar to that of Theorem~\ref{thm:harmonicmeasure} once one observes that for $z=a+\mathrm{i} b$, with $a\leq0$ and
$b\in\mathbb{R}$, $t>0$ and $x\in\mathbb{R}$ we have
\[
\mathrm{e}^{zt} = -\frac{1}{\pi}\int_{-\infty}^{\infty}\mathrm{e}^{\mathrm{i}
t x} \Re\Big[ \frac{1}{z-\mathrm{i}x}\Big] \mathrm{d} x =\frac{1}{\pi}%
\int_{-\infty}^{\infty}\mathrm{e}^{\mathrm{i} t x} \frac{-a}{a^{2} +
(b-x)^{2}} \mathrm{d} x.
\]
Therefore, by Fubini's theorem
\[
\mathrm{cov}(f(\xi_{0}),f(\xi_{t})) = \int_{\Re(z)\leq0} \mathrm{e}^{z t}
\nu(\mathrm{d} z) = \int_{x=-\infty}^{\infty}\mathrm{e}^{\mathrm{i} t x}
\Big[ \frac{-1}{\pi} \int_{\Re(z)\leq0} \Re\big(\frac{1}{z-\mathrm{i}%
x}\big) \nu(\mathrm{d} z) \Big] \mathrm{d}x.
\]
By letting $z=w\mathrm{i}$, where $w\in\mathbb{H}$, and the conformal
invariance of Brownian motion, one can immediately deduce that
\[
-\frac{1}{\pi} \Re\big(\frac{1}{z-\mathrm{i}x}\big) \mathrm{d}x,
\]
is the harmonic measure of Brownian motion in the left half-plane started at the point
$z$.
\end{proof}

\begin{proof}[Proof of Theorem~\ref{pr:NSCcts}.]
 First observe
that
\begin{align*}
\mathrm{var}(S_{T})  &  = 2\int_{s=0}^{T} (T-s) \int_{\Re(z)\leq0}
\mathrm{e}^{zx} \nu(\mathrm{d} z) \mathrm{d} s = 2 \int_{\Re(z)\leq0}
\Re\Big[ \frac{\mathrm{e}^{zT}-zT-1}{z^{2}}\Big] \nu(\mathrm{d} z).
\end{align*}
By splitting $S_{T}= \mathbb{E}_{0}(S_{T})+ S_{T}-\mathbb{E}_{0}(S_{T})$ we
obtain
\begin{align*}
\mathrm{var}(S_{T})  & = \mathbb{E}\big[(S_{T}-\mathbb{E}_{0}(S_{T}%
))^{2}\Big]+ \mathbb{E}\big[ \mathbb{E}_{0}(S_{T})^{2}\big]\\
& = \int_{\Re(z)\leq0} \Re\Big[ 2\frac{\mathrm{e}^{zT}-zT-1}{z^{2}} -
\frac{|1-\mathrm{e}^{zT}|^{2}}{|z|^{2}}\Big] \nu(\mathrm{d} z) + \int
_{\Re(z)\leq0} \frac{|1-\mathrm{e}^{zT}|^{2}}{|z|^{2}} \nu(\mathrm{d} z)=:
I_{1} +I_{2}.
\end{align*}
Assume $-\Re(1/z)\in L_{1}(\mathrm{d} \nu)$. Careful calculation shows that
\[
-\Re\big(\frac{1}{z}\big)\leftarrow\Re\Big[ 2\frac{\mathrm{e}^{zT}-zT-1}%
{z^{2}} - \frac{|1-\mathrm{e}^{zT}|^{2}}{|z|^{2}}\Big]
\leq\frac{C|x|}{x^{2}+y^{2}}\in L_{1}(\mathrm{d} \nu),
\]
and thus $I_{1}/T \to-\int_{\Re(z)\leq0} \Re(1/z)\nu(\mathrm{d} z)$. Fatou's
lemma also shows that if $\mathrm{var}(S_{T})/T$ converges then $\Re(1/z)\in
L_{1}(\mathrm{d} \nu)$.

Next we analyze $I_{2}/T$. Notice that
\[
|1-\mathrm{e}^{zT}|^{2} = (1-\mathrm{e}^{Tx})^{2}+ 4 \mathrm{e}^{Tx} \sin^{2}
(Ty/2),
\]
and since $(1-\mathrm{e}^{Tx})^{2}\leq T|x|$ for $x<0$, we have that
\begin{align*}
\frac{1}{T}\int_{\Re(z) \leq0} \frac{|1-\mathrm{e}^{zT}|^{2}}{|z|^{2}}
\nu(\mathrm{d} z)  & = \frac{1}{T}\int_{\Re(z)\leq0} \frac{4 \mathrm{e}%
^{Tx}\sin^{2}(Ty/2)}{x^{2} + y^{2}} \nu(\mathrm{d} x, \mathrm{d} y) + o(1)\\
& =\frac{1}{T}\int_{\Re(z) \leq0} \frac{4\sin^{2}(Ty/2)}{x^{2} + y^{2}}
\nu(\mathrm{d} x, \mathrm{d} y) + o(1).
\end{align*}
\noindent For $a>0$ write
\[
D_{a}:=\{x+\mathrm{i}y:0\leq-x\leq a,0\leq|y|\leq a\}.
\]
Notice that on $D_{a}^{(c)}$ the integrand is less than $1/2a^{2}$ and
therefore
\[
\int_{\Re(z)\leq0}\frac{4\sin^{2}(Ty/2)}{x^{2}+y^{2}}\nu(\mathrm{d}%
x,\mathrm{d}y)=\int_{D_{a}}\frac{4\sin^{2}(Ty/2)}{x^{2}+y^{2}}\nu
(\mathrm{d}x,\mathrm{d}y)+o(1).
\]
Let
\[
\epsilon_{T}:=\int_{D_{a}}\Big|\frac{\sin(Ty/2)}{Ty}\Big|\times\frac
{|x|\nu(\mathrm{d}x,\mathrm{d}y)}{x^{2}+y^{2}}\rightarrow0,
\]
since the bounded integrand vanishes and $\mu(\mathrm{d}z):=|x|\nu
(\mathrm{d}z)/(x^{2}+y^{2})$ is a finite measure. Let 
{$\delta_{T}
:=\max(\mathrm{e}^{-T},\sqrt{\epsilon_{T}})$}, 
so that $\delta_T>0$ and $\epsilon_{T}/\delta_{T}\rightarrow0$,
and define
\begin{gather*}
U_{a}:=\{x+\mathrm{i}y:0\leq-x\leq|y|\leq a\},\quad D_{a,T}:=\{x+\mathrm{i}%
y:0\leq\delta_{T}|y|\leq-x,|y|\leq a\}\\
U_{a,T}:=\{x+\mathrm{i}y:0\leq-x\leq\delta_{T}|y|,|y|\leq a\}
\end{gather*}
Since on $D_{a,T}$ we have $|y|\leq|x|/\delta_{T}$
\[
\frac{1}{T}\int_{D_{a,T}}\frac{\sin^{2}(Ty/2)}{x^{2}+y^{2}}\mathrm{d}\nu
\leq\frac{1}{\delta_{T}}\int_{D_{a,T}}\frac{|\sin(Ty/2)|}{Ty}\frac{|x|}%
{x^{2}+y^{2}}\mathrm{d}\nu\leq\frac{\epsilon_{T}}{\delta_{T}}\rightarrow0,
\]
and thus since $U_{a,T}=D_{a}-D_{a,T}$ it follows that $y^{2}+x^{2}=y^{2}%
(1+O(\delta_{T}^{2}))$ and
\[
\frac{1}{T}\int_{D_{a}}\frac{4\sin^{2}(Ty/2)}{x^{2}+y^{2}}\nu(\mathrm{d}%
x,\mathrm{d}y)=\frac{1}{T}\int_{U_{a,T}}\frac{\sin^{2}(Ty/2)}{(y/2)^{2}}%
\nu(\mathrm{d}x,\mathrm{d}y)\times(1+o(1))+o(1).
\]
Notice that on $U_{a}\cap D_{a,T}$ we have $0\leq\delta_{T}|y|\leq|x|\leq|y|\leq a$ and thus
\[
\frac{|x|}{x^{2}+y^{2}}\geq\frac{|x|}{x^{2}+y^{2}}\geq\frac{\delta_{T}|y|}%
{2y^{2}}=\frac{\delta_{T}}{2|y|}\Longrightarrow\frac{1}{|y|}\leq\frac
{1}{\delta_{T}}\frac{2|x|}{x^{2}+y^{2}}.
\]
Since $U_{a,T}=U_{a}-U_{a}\cap D_{a,T}$, from the above
\[
\frac{1}{T}\int_{U_{a,T}}\frac{\sin^{2}(Ty/2)}{(y/2)^{2}}\nu(\mathrm{d}%
x,\mathrm{d}y)=\frac{1}{T}\int_{U_{a}}\frac{\sin^{2}(Ty/2)}{(y/2)^{2}}%
\nu(\mathrm{d}x,\mathrm{d}y)+o(1).
\]
Therefore
\begin{multline*}
\frac{1}{T}\int_{U_{a}\cap D_{a,T}}\!\!\!\!\frac{\sin^{2}(Ty/2)}{(y/2)^{2}}%
\nu(\mathrm{d}x,\mathrm{d}y)
\leq C\int_{U_{a}\cap D_{a,T}}\!\!\!\!\frac{|\sin
(Ty/2)|}{T|y|}\times\frac{1}{|y|}\nu(\mathrm{d}x,\mathrm{d}y)\\
\leq\frac{C}{\delta_{T}}\int_{U_{a}\cap D_{a,T}}\!\!\!\!
\frac{|\sin(Ty/2)|}%
{T|y|}\times\frac{|x|}{x^{2}+y^{2}}\nu(\mathrm{d}x,\mathrm{d}y)\leq
C\frac{\epsilon_{T}}{\delta_{T}}\rightarrow0.
\end{multline*}
Finally let
\[
U:=\{x+\mathrm{i}y:0\leq-x\leq|y|\}.
\]
From the above computation, since $a>0$ was arbitrary and all error terms
depending on $a$ vanish as $a\rightarrow\infty$ we have
\begin{multline*}
\lim_{T\rightarrow\infty}\frac{1}{T}\int_{\Re(z)\leq0}\frac{|1-\mathrm{e}%
^{zT}|^{2}}{|z|^{2}}\nu(\mathrm{d}z)=\lim_{T\rightarrow\infty}\frac{1}{T}%
\int_{U_{a}}\frac{\sin^{2}(Ty/2)}{(y/2)^{2}}\nu(\mathrm{d}x,\mathrm{d}y)\\
=\lim_{T\rightarrow\infty}\frac{1}{T}\int_{U}\frac{\sin^{2}(Ty/2)}{(y/2)^{2}%
}\nu(\mathrm{d}x,\mathrm{d}y)=\lim_{T\rightarrow\infty}\frac{1}{T}\int
_{x=0}^{\infty}\frac{\sin^{2}(Tx/2)}{(x/2)^{2}}G(\mathrm{d}x),
\end{multline*}
where $G(x)=\nu(U_{x})$, using similar arguments to the proof of
Theorem~\ref{pr:NSC}. Now we have
\begin{align*}
\frac{1}{T}\int_{x=0}^{\infty}\frac{\sin^{2}(Tx/2)}{(x/2)^{2}}G(\mathrm{d}x)
&  =\frac{1}{T}\int_{x=0}^{\pi}\frac{\sin^{2}(Tx/2)}{\sin^2(x/2)}G(\mathrm{d}%
x)+O(1/T),
\end{align*}
since
\[
\frac{1}{T}\int_{x=0}^{\pi}\sin^{2}(Tx/2)\Big|\frac{1}{(\tfrac{x}{2})^{2}%
}-\frac{1}{\sin^{2}(\tfrac{x}{2})}\Big|G(\mathrm{d}x)\leq\frac{1}{T}\int
_{x=0}^{\pi}\frac{(x/2)^{4}}{\sin^{2}(x/2)(x/2)^{2}}G(\mathrm{d}x)=
O(1/T).
\]
\noindent The result then follows from Theorem A.
\end{proof}

\appendix

\section{Technical lemmas}

\subsection{Standard Tauberian Theorems}

In order to make this paper more self-contained we state the following
classical Tauberian theorems (Theorem 2.3 In \cite{Seneta} or Theorem 1.7.
in \cite{Bing}, due to Feller).

\begin{theorem}
\label{t1} Let $U(x)$ be a monotone non-decreasing function on $[0$,$\infty)$
such that%
\begin{equation*}
w(x)=\int_{0^{-}}^{\infty}e^{-xu}\mathrm{d}U(u)\text{ is finite for all }%
x>0. 
\end{equation*}
Then if $\rho\geq0$ and $L$ is a slowly varying function, then 
\begin{equation*}
w(x)=x^{-\rho}L(x)\text{ as }x\rightarrow\infty\text{ \ iff \ }U(x)=x^{\rho
}L(1/x)/\Gamma(\rho+1)\text{ as }x\rightarrow0^{+}. 
\end{equation*}
\end{theorem}

We give the monotone Tauberian theorem (Theorem 2.4 in \cite{Seneta} or
Theorem 2.4 in \cite{Bing})

\begin{theorem}
\label{thm:tm} Let $U(x)$ defined and positive on $[A,\infty)$ for some $A>0$
given by 
\begin{equation*}
U(x)=\int_{A}^{x}u(y)\mathrm{d}y,\text{ }
\end{equation*}
where $u(y)$ is ultimately monotone. Then if $\rho\geq0$ and $L$ is a slowly
varying function, then%
\begin{equation*}
U(x)=x^{\rho}L(x)\text{ as }x\rightarrow\infty\text{ implies }u(x)\sim\rho
x^{\rho-1}L(x)\text{ as }x\rightarrow\infty. 
\end{equation*}
If $\rho>0$ then $u(x)$ is regularly varying.
\end{theorem}

\subsection{Auxiliary Lemma for Theorem~\protect\ref{thm:MH}}

\begin{lemma}
\label{lem:aux} For $V$ and $H$ as defined in Theorem~\ref{thm:MH}, we have 
\begin{equation*}
\frac{H(1/x)}{2\theta V(x)}\rightarrow1\text{ as }x\rightarrow0^{+}. 
\end{equation*}
In particular if $V(x)$ is slowly varying at $0$, $H(1/x)$ is slowly varying
at $\infty$.
\end{lemma}

\noindent\textbf{Proof.} Let $r$ be the measure defined in (\ref{def r}). By
Definition~(\ref{tail tau}) 
\begin{equation*}
\mathbb{P}(\tau_{1}>u)\mathbb{=}2\theta\int_{0}^{1}x^{u}(1-x)\mu (\mathrm{d}%
x)=\theta\int_{0}^{1}(1-y)^{u}yr(\mathrm{d}y). 
\end{equation*}
We show first for any $\delta\in(0,1)$ 
\begin{equation}
\int_{0}^{1}(1-y)^{u}r(\mathrm{d}y)=O(u^{\delta-3})+\int_{0}^{1}e^{-uy}yr(%
\mathrm{d}y).   \label{tail2}
\end{equation}
To prove this, notice that for $u\geq0,$ $0\leq y\leq1,$ for positive $t,m$
and some $C_{m}$ 
\begin{equation*}
|(1-y)^{u}-\mathrm{e}^{-u y}|\leq|1-y-\mathrm{e}^{-y}|u\mathrm{e}%
^{-(u-1)y}\text{ and }\mathrm{e}^{t}\geq 1+t^{m}/C_{m}.\text{ }
\end{equation*}
Then for any $\delta\in(0,1)$ 
\begin{equation*}
\int_{0}^{1}\mathrm{e}^{-uy}uy^{3}r(\mathrm{d}y)
\leq C\int_{0}^{1}\frac{r(\mathrm{d}y)}{u^{3-\delta}y^{1-\delta}}
\leq Cu^{\delta-3}\int_{0}^{1}\frac{r(\mathrm{d}y)}{y^{1-\delta}}, 
\end{equation*}
and after some rearrangement%
\begin{equation*}
\int_{0}^{1}\frac{r(\mathrm{d}y)}{y^{1-\delta}}=C\int_{y=1}^{\infty}\frac{%
V(1/y)}{y^{1+\delta}}\mathrm{d}y<\infty, 
\end{equation*}
since $V(u):=\int_{u}^{1}r(\mathrm{d}y)/y$, is slowly varying as $%
u\rightarrow0^{+}$. Since for $u\geq0$ 
\begin{equation*}
H(u)=\mathbb{E}\tau_{1}^{2}I(\tau_{1}<u)=2\int_{0}^{u}x\mathbb{P}(\tau
_{1}>x)\mathrm{d}x-u^{2}\mathbb{P}(\tau_{1}>u), 
\end{equation*}
by Fubini's theorem and (\ref{tail2}) we derive 
\begin{align*}
H(u) & =O(u^{\delta-1})-\theta u^{2}\int_{0}^{1}e^{-uy}yr(\mathrm{d}%
y)+2\theta\int_{0}^{1}\frac{1-e^{-yu}}{y}r(\mathrm{d}y)-2\theta
u\int_{0}^{1}e^{-yu}r(\mathrm{d}y) \\
& =:O(u^{\delta-1})-\theta I_{1}(u)+2\theta I_{2}(u)-4I_{3}(u).
\end{align*}
Also note that by integration by parts, 
\begin{equation*}
r(z)=\int_{0}^{z}r(\mathrm{d}s)=-\int_{0}^{z}yV(\mathrm{d}%
y)=\int_{0}^{z}[V(y)-V(z)]\mathrm{d}y. 
\end{equation*}
Then $R$ can also be written as 
\begin{equation*}
r(z)=\int_{1/z}^{\infty}\frac{\mathrm{d}V(1/y)}{y}, 
\end{equation*}
and since $V(1/y)$ is slowly varying as $y\rightarrow\infty$ we have by
Theorem~1.6.5 in \cite{Bing} that 
\begin{equation*}
\frac{r(z)}{zV(z)}\rightarrow0,\quad\text{as $z\rightarrow0^{+}$.}
\end{equation*}

Now let $K$ be an arbitrary positive number. Since we will first take limits
as $x\rightarrow 0^{+}$, we can assume that $x$ is small enough so that $Kx<1
$. Therefore splitting the integral $\int_{0}^{1}=\int_{0}^{Kx}+\int_{Kx}^{1}
$ and applying standard analysis, we derive 
\begin{align*}
I_{2}(u)& =\int_{y=0}^{Kx}\frac{1-e^{-yu}}{y}r(\mathrm{d}y)+\int_{y=Kx}^{1}%
\frac{1-e^{-yu}}{y}r(\mathrm{d}y) \\
& =K\frac{r(Kx)}{Kx}+V(Kx)(1+O(\mathrm{e}^{-Kux})).
\end{align*}%
Then for $u=1/x$, arbitrary $K>0$, since for each fixed $K$ 
\begin{equation*}
K\frac{r(Kx)}{KxV(x)}\rightarrow 0\quad \mbox{and}\quad \frac{V(Kx)}{V(x)}%
\rightarrow 1,\text{ as }x\rightarrow 0^{+}
\end{equation*}%
we have 
\begin{equation*}
\limsup_{x\rightarrow 0^{+}}\Big|\frac{I_{2}(1/x)}{V(x)}-1\Big|\leq C\mathrm{%
e}^{-K},
\end{equation*}%
and since $K>0$ is arbitrary $I_{2}(1/x)/V(x)\rightarrow 1$, as $%
x\rightarrow 0^{+}$ and $u=1/x$.

Finally let again $K>0$ be arbitrary. Then 
\begin{equation*}
\sup_{t>K}e^{-t}[1+t+t^{2}]=W_{K},\quad W=\sup_{K>0}W_{K}
\end{equation*}
and notice that $W_{K}\rightarrow0$ as $K\rightarrow\infty$. Then we have
for fixed arbitrary $K>0$, and $x$ small enough so that $Kx<1$, 
\begin{align*}
I_{1}+I_{3} & =\int_{y=0}^{1}e^{-yu}(uy+(uy)^{2})\frac{r(\mathrm{d}y)}{y}%
\leq(1+W)u\int_{y=0}^{Kx}dr(y)+W_{K}\int_{y=Kx}^{1}\frac{r(\mathrm{d}y)}{y}
\\
& =(1+W)\frac{r(Kx)}{x}+W_{K}V(Kx).
\end{align*}
Therefore for $u=1/x$, as above we have 
\begin{equation*}
\limsup_{x\rightarrow0^{+}}\Big|\frac{I_{1}(x^{-1})+I_{3}(x^{-1})}{V(x)}\Big|%
\leq W_{K}, 
\end{equation*}
and since $K$ is arbitrary the claim follows from the fact that 
\begin{equation*}
\frac{I_{1}(x^{-1})+I_{3}(x^{-1})}{V(x)}\rightarrow0,\text{ as }%
x\rightarrow0^{+}. 
\end{equation*}
\hfill$\square$

\end{document}